\documentclass[a4paper,reqno]{amsart}                

\usepackage[latin1]{inputenc}                  
\usepackage[T1]{fontenc}                       
\usepackage[spanish,english]{babel}            
\usepackage{amsmath,amssymb,amsthm}            
\usepackage{latexsym}                          
\usepackage{delarray}                          
\usepackage{bbm}                               
\usepackage{hyperref}                          
\usepackage[pdftex,usenames,dvipsnames]{color} 
\usepackage{bbding,trfsigns}                   
\usepackage{wasysym}                           
\usepackage{datetime}                          

\DeclareMathAlphabet{\mathpzc}{OT1}{pzc}{m}{it} 

\newtheorem{Th}{Theorem}[section]              

\newtheorem{Lem}{Lemma}[section]

\newcommand{\N}{\mathbb{N}}

\newcommand{\R}{{\mathbb{R}}}

\newcommand{\W}{\mathbb{W}}

\newcommand{\fa }{\frak{a}(t)}
\newcommand{\fb }{\frak{b}(t)}

\newcommand{\eps}{\varepsilon}

\DeclareMathOperator{\supp}{supp}

\title[Riesz transforms and multipliers for the Bessel-Grushin operator]
      {Riesz transforms and multipliers for the Bessel-Grushin operator}

\author[V. Almeida]{V. Almeida}
\author[J.J. Betancor]{J.J. Betancor}
\author[A.J. Castro]{A.J. Castro}
\author[K. Sadarangani]{K. Sadarangani}

\address{\newline
        Víctor Almeida, Jorge J. Betancor, Alejandro J. Castro \newline
        Departamento de Análisis Matemático,
        Universidad de la Laguna, \newline
        Campus de Anchieta, Avda. Astrofísico Francisco Sánchez, s/n, \newline
        38271, La Laguna (Sta. Cruz de Tenerife), Spain}
\email{valmeida@ull.es, jbetanco@ull.es, ajcastro@ull.es}

\address{\newline
        Kishin Sadarangani \newline
        Departamento de Matemáticas,
        Universidad de Las Palmas de Gran Canaria, \newline
        Campus de Tafira Baja, \newline
        35017, Las Palmas de Gran Canaria, Spain}
\email{ksadaran@dma.ulpgc.es}



\thanks{The second and the third authors are partially supported by MTM2010/17974.
        The third author is also supported by a FPU grant from the Government of Spain.}

\begin{document}

  \footnotetext{Date: \today.}

  \maketitle                

\begin{abstract}
    We establish that the spectral multiplier $\frak{M}(G_{\alpha})$ associated to the differential operator
    $$ G_{\alpha}=- \Delta_x +\sum_{j=1}^m{{\alpha_j^2-1/4}\over{x_j^2}}-|x|^2 \Delta_y \; \text{ on } (0,\infty)^m \times \R^n,$$
    which we denominate Bessel-Grushin operator,
    is of weak type $(1,1)$ provided that $\frak{M}$ is in a suitable local Sobolev space.
    In order to do this we prove a suitable weighted Plancherel estimate.
    Also, we study $L^p$-boundedness properties of Riesz transforms associated to $G_{\alpha}$, in the case $n=1$.
\end{abstract}

\section{Introduction} \label{sec:intro}

If $m,n\in\N$ the Grushin operator on $\R^m\times \R^n$ is defined by
$$G
    =-\sum_{j=1}^m{{\partial^2}\over{\partial x_j^2}} -\left(\sum_{j=1}^m{x_j^2}\right)\left(\sum_{j=1}^n{{\partial^2}\over{\partial y_j^2}}\right).$$
In this paper we consider the operators we call Bessel-Grushin
operators which appear when Bessel operators replace the first Laplacian operator in $G$. If $\beta >-1/2$ the Bessel operator is defined by
$$ B_{\beta}={{d^2}\over{dx^2}}-{{\beta^2-1/4}\over{x^2}}, \quad  x\in(0,\infty).$$

Let $m,n\in\N$ and $\alpha=(\alpha_1,...,\alpha_m)\in(-1/2,\infty)^m$.
We introduce the
Bessel-Grushin operator $G_{\alpha}$ as follows
$$G_{\alpha}
    =-\sum_{j=1}^m{B_{\alpha_j}} -\left(\sum_{j=1}^m{x_j^2}\right)\left(\sum_{j=1}^n{{\partial^2}\over{\partial y_j^2}}\right) ,\;\;\mbox{on}\;\;(0,\infty)^m\times\R^n.$$
Our objective in this paper is to study $L^p$-boundedness
properties of spectral multipliers and  Riesz transforms associated
with $G_{\alpha}$, being $n=1$ in the case of Riesz transforms.  We are motivated by the recent papers of
Chen and Sikora \cite{CS},
Jotsaroop, Sanjay and Thangavelu \cite{JST},
Martini and Muller \cite{MM}
and
Martini and Sikora \cite{MaSi},
about Grushin operators.

We consider the Laguerre operator
$$L_{\beta}
    =- {{d^2}\over{dx}} +{{\beta^2-1/4}\over{x^2}}+x^2,\;\;\mbox{on}\;\;(0,\infty).$$
where $\beta>-1/2$. We have that, for every $k\in\N$,
$$L_{\beta}{\varphi}_k^{\beta}
    =2(2k+\beta +1)\varphi_k^{\beta},$$
where
$$\varphi_k^{\beta}(x)=\left({{2\Gamma(k+1)}\over{\Gamma(k+\beta +1)}}\right)^{1/2}e^{-{x^2}/ 2}x^{\beta +{1/ 2}}l_{k}^{\beta}(x^2),\;\;x\in (0,\infty),$$
and $l_{k}^{\beta}$ represents the $k$-th Laguerre polynomial of order $\beta$ (\cite[p. 100-102]{Sz}).
The family $\left\{\varphi_k^{\beta}\right\}_{k\in\N}$ is an orthonormal basis in $L^2(0,\infty)$.

Let $\alpha=(\alpha_1,...,\alpha_m)\in(-1/2,\infty)^m$. We define the Laguerre operator $\frak{L}_\alpha$ on $(0,\infty)^m$ by
$$\frak{L}_\alpha=\sum_{j=1}^m{L_{\alpha_j}}.$$

If for $k=(k_1,...,k_m)\in\N^m$,
$$\Phi_k^\alpha(x)=\prod_{j=1}^m{\varphi_{k_j}^{\alpha_j}(x_j)},\;\;\;\;x=(x_1,...,x_m)\in(0,\infty)^m,$$
we have that
$$\frak{L}_\alpha\Phi_k^\alpha(x)=2(2\frak{s}(k)+\frak{s}(\alpha)+m)\Phi_k^\alpha(x).$$
where $\frak{s}(k)=k_1+...+k_m$ and $\frak{s}(\alpha)=\alpha_1+...+\alpha_m$.

The system $\{\Phi_k^\alpha(x)\}_{k\in\N^m}$ is an orthonormal basis in $L^2((0,\infty)^m)$.

We denote by $\mathcal{F}_2(f)$ the Fourier transform of $f\in L^2((0,\infty)^m\times\R^n)$ with respect to the second $\R^n$-variable, that is,

$$\mathcal{F}_2(f)(x,y)={1\over{(2\pi)^{n/2}}}\int_{\R^n}{e^{-iy\cdot z}f(x,z)dz},\;\;\;\;(x,y)\in(0,\infty)^m\times\R^n.$$

The Bessel-Grushin operator $\widetilde G_{\alpha}$ is defined by
$$\widetilde G_{\alpha}(f)
    =\mathcal{F}_2^{-1}\left[\sum_{k\in\N^m}{2(2\frak{s}(k)+\frak{s}(\alpha)+m)|u|c_k^{\alpha}(\mathcal{F}_2(f);u)}\Phi_k^{\alpha}(x;|u|)\right],\;\;f\in D(\widetilde G_{\alpha}),$$
where $\mathcal{F}_2^{-1}$ represents the inverse map of $\mathcal{F}_2$ and
$$D(\widetilde G_{\alpha})
    =\left\{f\in L^2((0,\infty)^m\times\R^n):\;\sum_{k\in\N^m}^\infty{2(2\frak{s}(k)+\frak{s}(\alpha) +m)|u|c_k^{\alpha}(\mathcal{F}_2(f);u)}\Phi_k^{\alpha}(x;|u|) \in L^2((0,\infty)^m\times\R^n)\right\}.$$
Here, for every $k\in\N^m$ and $g\in L^2((0,\infty)^m\times\R^n)$,
$$c_k^{\alpha}(g;u)=\int_{(0,\infty)^m}{\Phi_k^{\alpha}(x;|u|)g(x,u)dx}, \quad u\in\R^n,$$
where
$$\Phi_k^{\alpha}(x;a)
    =a^{m\over 4}\Phi_k^{\alpha}(\sqrt{a}x),\;\;x\in (0,\infty)^m\;\mbox{and}\;a>0.$$

Note that, for every $a>0$, $\left\{\Phi_k^{\alpha}(\cdot;a)\right\}_{k\in\N^m}$ is an orthonormal basis in $L^2((0,\infty)^m)$.
Moreover, denoting by $\frak{L}_{\alpha}(a)$ the operator
$$\frak{L}_{\alpha}(a)
    =\sum_{j=1}^m{(-B_{\alpha_j}+a^2x_j^2)}, \;\;\;\; a>0,$$
we have that
$$\frak{L}_{\alpha}(a)\Phi_k^{\alpha}(\cdot,a)
    =2(2\frak{s}(k)+\frak{s}(\alpha) +m)a\Phi_k^{\alpha}(\cdot;a), \quad k\in\N^m\;\mbox{and}\;a>0.$$
If $f\in C_c^\infty((0,\infty)^m\times\R^n)$, the space of smooth functions with compact support in $(0,\infty)^m \times \R^n$,
 it is clear that $\widetilde G_{\alpha}f=G_{\alpha}f$.
In the sequel we write also $G_{\alpha}$ to refer us to the operator $\widetilde  G_{\alpha}$.

Suppose that $\frak{M}$ is a bounded Borel function on $(0,\infty)$. We define the spectral multiplier $\frak{M}(G_{\alpha})$ associated with $\frak{M}$ by
$$\frak{M}(G_{\alpha})f
    =\mathcal{F}_2^{-1}\left[\sum_{k\in\N^m}^\infty{\frak{M}\big(2(2\frak{s}(k)+\frak{s}(\alpha) +m)|u|\big)c_k^{\alpha}(\mathcal{F}_2(f);u)}\Phi_k^{\alpha}(x;|u|)\right], \quad f\in L^2((0,\infty)^m\times\R^n).$$

Since $\frak{M}$ is bounded the operator $\frak{M}(G_{\alpha})$ is bounded from $L^2((0,\infty)^m\times\R^n)$ into itself.
As usual, the question is to give conditions on the function $\frak{M}$ such that the operator $\frak{M}(G_{\alpha})$
can be extended from $L^2((0,\infty)^m\times\R^n)\cap L^p((0,\infty)^m\times\R^n)$ to $L^p((0,\infty)^m\times\R^n)$
as a bounded operator from $L^p((0,\infty)^m\times\R^n)$ into itself, when $p\neq 2$.

In the classical H\"ormander multiplier, local Sobolev norms are considered to describe smoothness of $\frak{M}$ in order to get
$L^p$-boundedness of the multiplier operator. These arguments have been used by
Christ \cite{Ch},
Duong, Ouhabaz and Sikora \cite{DOS},
Duong, Sikora and Yen \cite{DSY},
Hebisch \cite{He},
Hulanicki and Stein \cite[cf.]{FoSt},
in different settings.

If $1\leq q\leq\infty$ and $s>0$, we denote by $W_q^s(\R)$ the $L^q$-Sobolev space of order $s$.
We choose $\eta\in C_c^\infty(0,\infty)$ not identically zero. The "local" $W_q^s$ norm $\| \cdot \|_{MW_q^s}$ is defined by
$$\|\frak{M}\|_{MW_q^s}=\sup_{t>0}\|\eta\delta_t\frak{M}\|_{W_q^s},$$
where $\delta_t\frak{M}(s)=\frak{M}(ts),\;t,s\in (0,\infty)$. When we consider different functions $\eta$ we get equivalent local Sobolev norms.

We now establish our result about spectral multipliers for Bessel-Grushin operators.

\begin{Th}\label{Th1.1}
    Let $\alpha\in[1/ 2,\infty)^m$, and $D=\max\{m+n,2n\}$. Suppose that $\frak{M}$ is a bounded Borel measurable function on $(0,\infty)$ such that $\|\frak{M}\|_{MW_2^s}<\infty$ for $s>D/2$.
    Then, the spectral multiplier $\frak{M}(G_{\alpha})$ can be extended from $L^1((0,\infty)^m\times\R^n)\cap L^2((0,\infty)^m\times\R^n)$ to
    $L^1((0,\infty)^m\times\R^n)$ as a bounded operator from $L^1((0,\infty)^m\times\R^n)$ into $L^{1,\infty}((0,\infty)^m\times\R^n)$.
\end{Th}

Note that when $\frak{M}$ is a bounded measurable function on $(0,\infty)$, $\frak{M}(G_{\alpha})$ is bounded from
$L^2((0,\infty)^m\times\R^n)$ into itself, and then classical interpolation theorems, duality and Theorem \ref{Th1.1} imply that the ope\-rator
$\frak{M}(G_{\alpha})$ can be extended from $L^p((0,\infty)^m\times\R^n)\cap L^2((0,\infty)^m\times\R^n)$ to $L^p((0,\infty)^m\times\R^n)$
as a bounded ope\-rator from $L^p((0,\infty)^m\times\R^n)$ into it self, for every $1<p<\infty$, provided that $\frak{M}$
satisfies the conditions in Theorem \ref{Th1.1}.

According to \cite[Theorem 1]{MM} we conjeture that when $m<n$ the result in Theorem~\ref{Th1.1} is also true for
$s > (m+n)/2$. Moreover, we also conjeture that the order $(m+n)/2$ of differentiability in Theorem~\ref{Th1.1}
cannot be decreased. In order to show this the idea is to work with the imaginary power $G_\alpha^{it}$, $t \in \mathbb{R}$,
of the Bessel-Grushin operator and to adapt the arguments in \cite{SikWr} (see \cite[Section 5]{CS} and \cite[Section 5]{MaSi}).
We will study these questions in a foregoing paper.

The key result in the proof of Theorem \ref{Th1.1} is a weighted Plancherel type estimate.

Next we introduce Riesz transforms, when $n=1$, associated with Bessel-Grushin operators. Let $\beta>-1/2$ and $a>0$. We define
\begin{equation}\label{3.1}
    A_{\beta}(a)
        ={d\over{dx}}+ax - \frac{\beta +{1/ 2}}{x}\;\;\mbox{and}\;\;A_{\beta}^*(a)=-{d\over{dx}}+ax- \frac{\beta+{1/ 2}}{x},\;\;\;x\in(0,\infty).
\end{equation}
Note that $A_{\beta}^*(a)$ is the "formal" adjoint of $A_{\beta}(a)$ in $L^2(0,\infty)$. We have that
$$-B_{\beta}+a^2x^2
    =A_{\beta}^*(a)A_{\beta}(a)+2a(\beta +1),\;\;\;\;a,x\in(0,\infty).$$
This decomposition suggests to "formally" define  the Riesz transforms for the scaled Laguerre operator $\frak{L}_{\alpha}(a)$, $\alpha=(\alpha_1,...,\alpha_m)\in(-1/2,\infty)^m$ and $a>0$, as follows: for every $j=1,...,m$

\begin{equation}\label{I1}
R_{\alpha,j}(a)
    =A_{\alpha_j}(a)\frak{L}_{\alpha}^{-{1/2}}(a)
    \end{equation}
  and
  \begin{equation}\label{I2}
\widetilde {R}_{\alpha,j}(a)
    =A_{\alpha_j}^*(a)\frak{L}_{\alpha}^{-{1/2}}(a)
    \end{equation}
According to some well-known properties of Laguerre functions (see, for instance, \cite[(2.17) and (2.18), p. 1004]{HTV2} and \cite[p. 406]{NoSt1}) we have that, for every $k=(k_1,...,k_m)\in\N^m$, $\alpha=(\alpha_1,...,\alpha_m)\in(-1/2,\infty)^m$,  $a>0$ and $j=1,...,m$,

$$A_{\alpha_j}(a)\Phi_k^{\alpha}(x;a)
    =-2\sqrt{k_ja}\Phi_{k-e_j}^{\alpha +e_j}(x;a),\;\;x\in (0,\infty)^m,$$
and
$$A_{\alpha_j}^*(a)\Phi_k^{\alpha}(x;a)
    =-2\sqrt{(k_j+1)a}\Phi_{k+e_j}^{\alpha -e_j}(x;a),\;\;x\in (0,\infty)^m.$$
Here $e_j=(e_j^1,...,e_j^m)$, where $e_j^i=\left\{\begin{array}{l} 0\;,\;\;\;i\neq j \\ 1\;,\;\;\;i=j \end{array}\right.$, $j=1,...,m$, and  we understand  $\Phi_{-1}^{\beta}=0$, for every $\beta>-1/2$.

For every $\gamma >0$ and $\alpha\in(-1/2,\infty)^m$, the $-\gamma$-th power $\frak{L}_{\alpha}^{-\gamma}(a)$ is defined by
$$\frak{L}_{\alpha}^{-\gamma}(a)f
    =\sum_{k\in\N^m}{{{c_k^{\alpha}(a)(f)}\over{(2(2\frak{s}(k)+\frak{s}(\alpha) +m)a)^{\gamma}}}\Phi_k^{\alpha}(\cdot;a)},\;\;f\in L^2((0,\infty)^m),$$
where, for every $k\in\N^m$,
$$ c_k^{\alpha}(a)(f)=\int_{(0,\infty)^m}{\Phi_k^{\alpha}(x;a)f(x)dx},\;\;f\in L^2((0,\infty)^m).$$

We define the Riesz transforms on $L^2((0,\infty)^m)$ (according to (\ref{I1}) and (\ref{I2})) associated with
$\frak{L}_\alpha(a)$, $\alpha\in(-1/2,\infty)^m$ and $a>0$ as follows: for every $j=1,2,...,m$,
$$R_{\alpha,j}(a)f
    =-2\sum_{k\in\N^m}{\sqrt{k_j}{{c_k^{\alpha}(a)(f)}\over{\sqrt{2(2\frak{s}(k)+\frak{s}(\alpha) +m)}}}\Phi_{k-e_j}^{\alpha+e_j}(\cdot;a)},\;\;f\in L^2((0,\infty)^m),$$
and
$$\widetilde {R}_{\alpha,j}(a)f
    =-2\sum_{k\in\N^m}{\sqrt{k_j+1}{{c_k^{\alpha}(a)(f)}\over{\sqrt{2(2\frak{s}(k)+\frak{s}(\alpha) +m)}}}\Phi_{k+e_j}^{\alpha-e_j}(\cdot;a)},
    \quad f\in L^2((0,\infty)^m), \ \alpha_j >1/2.$$

Note that, in virtue of Plancherel equality for Laguerre function spaces, we deduce that, for every $j=1,...,m$,  $R_{\alpha,j}(a)$ and $\widetilde {R}_{\alpha,j}(a)$
are bounded operators from $L^2((0,\infty)^m)$ into itself. Moreover, if $f\in span \{\Phi_k^{\alpha}(\cdot;a)\}_{k\in\N^m}$, the linear space generated  by $\{\Phi_k^{\alpha}(\cdot;a)\}_{k\in\N^m}$ , then $R_{\alpha,j}(a)f
    =A_{\alpha_j}(a)\frak{L}_{\alpha}^{-{1/2}}(a)f\;\;\mbox{and}\;\;\widetilde {R}_{\alpha,j}(a)f=A_{\alpha_j}^*(a)\frak{L}_{\alpha}^{-{1/2}}(a)f$ $j=1,...,m$. $L^p$-boundedness properties of Riesz transforms associated with Laguerre function
expansions have been established in  \cite{HTV2} and \cite{NoSt1}, among others.

The above comments suggest to define Riesz transforms $R_\alpha,j$ and $\widetilde{R}_\alpha,j$ in Bessel-Grushin settings as follows
$$R_{\alpha,j}(f)(x,y)
    =\mathcal{F}_2^{-1}(R_{\alpha,j}(|u|)(\mathcal{F}_2(f))(x,u))(y),\;\;\;f\in L^2((0,\infty)^m\times\R),$$
and
$$\widetilde {R}_{\alpha,j}(f)(x,y)
    =\mathcal{F}_2^{-1}(\widetilde {R}_{\alpha,j}(|u|)(\mathcal{F}_2(f))(x,u))(y),\;\;\;f\in L^2((0,\infty)^m\times\R),$$
where $\alpha\in(-1/2,\infty)^m$ and $j=1,...,m$. Then, Plancherel theorem for Fourier transform implies that $R_{\alpha,j}$ and $\widetilde{R}_{\alpha,j}$, $j=1,...,m$, are bounded operators from $L^2((0,\infty)^m\times\R^n)$ into itself.

We prove the following result.

\begin{Th}\label{Th1.2}
    Let $\alpha\in[1/2,\infty)^m$, $j=1,...,m$ and $1<p<\infty$. Then, the Riesz transforms $R_{\alpha,j}$ and $\widetilde{R}_{\alpha,j}$ are bounded operators from $L^p((0,\infty)^m\times\R)$ into itself.
\end{Th}

In order to prove this theorem we start using the main idea in the proof of \cite[Theorem 1.1]{JST},
namely, we see $R_{\alpha,j}$ and $\widetilde{R}_{\alpha,j}$, $j=1,...,m$, as Banach valued Fourier multipliers and then we use
the celebrated Weis' multiplier result \cite[Theorem 3.4]{We}. But to show that the R-boundedness properties hold for the
family of operators which appear in the Bessel-Grushin context, we can not proceed as in \cite[Section 2]{JST}
because Laguerre functions have not as nice operational properties  as Hermite functions. Roughly speaking
we take advantage that the operators we need to study are bounded perturbations of those operator handled in \cite{JST}.

Throughout this paper we always denote by $c$ and $C$ positive constants that can change from one line to the other.

\section{Proof of Theorem \ref{Th1.1}} \label{sec:2}

The strategy of the proof of this theorem is the same as in \cite{MaSi} (see also \cite{CS}, \cite{MM} and \cite{Sik})
and the key result is a weighted Plancherel inequality. Laguerre expansions play an important role and we need to get estimations
involving Laguerre functions.

The Bessel-Grushin operator $G_{\alpha}$ is selfadjoint and positive in $L^2((0,\infty)^m\times\R^n)$.
Then, $-G_{\alpha}$ generates a semigroup of contractions $\{e^{-tG_{\alpha}}\}_{t >0}$ in $L^2((0,\infty)^m\times\R^n)$.
Moreover, since $G_{\alpha}-G=\sum_{j=1}^m{(\alpha_j^2-1/4)/x_j^2}$, by using the perturbation formula (see \cite[Corollary 1.7, p. 161]{EN})  we get
\begin{equation}\label{eqB.1}
    e^{-tG}f-e^{-tG_{\alpha}}f
        =\int_0^t{e^{-(t-s)G}}\;\sum_{j=1}^m{{{\alpha_j^2-1/4}\over{x_j^2}}\;e^{-sG_{\alpha}}f\;ds}.
\end{equation}
Here and in the sequel, we identify each measurable function $f$ on $(0,\infty)^m \times \R^n$ with the function $f_0$ defined by
$$f_0(x,y)=\left\{
    \begin{array}{rl}
        f(x,y), & x\in(0,\infty)^m, \\
        0, & x\notin \R^m\setminus\{(0,\infty)^m\}.
    \end{array}\right., \quad  y \in \R^n.$$
From \eqref{eqB.1} we deduce that $e^{-tG_{\alpha}}f\leq e^{-tG}f$, $0\leq f\in L^2((0,\infty)^m\times \R^n)$ and $t >0$.

According to \cite[Proposition 3]{MaSi} there exists a distance $\rho$ in ${\R}^m\times\R^n$ such that the triple
$({\R}^m\times\R^n,\rho ,|\cdot|)$, where $|\cdot|$ denotes the Lebesgue measure in ${\R}^m\times\R^n$, is a homogeneous type space
(in the sense of Coifman and Weiss \cite{CW}), and that
\begin{equation*}\label{eqB.2}
    0\leq \W_t((x_1,y_1),(x_2,y_2))
        \leq C \frac{e^{-c\rho((x_1,y_1),(x_2,y_2))^2/t}}{|B_\rho((x_2,y_2),\sqrt{t})|}, \quad  (x_j,y_j)\in{\R}^m\times\R^n,\;j=1,2,\;\mbox{and}\;t >0,
\end{equation*}
where, for every $t>0$, $\W_t$ represents the integral kernel of $e^{-tG}$.

Hence, for every $t>0$, the operator $e^{-tG_{\alpha}}$ is bounded from $L^1((0,\infty)^m\times\R^n)$ into
$L^q((0,\infty)^m\times\R^n)$, $1\leq q <\infty$. Then, for every $t>0$,
\begin{equation}\label{eqB.3}
    e^{-tG_{\alpha}}(f)(x_1,y_1)
        =\int_{(0,\infty)^m\times\R^n}{\W_t^\alpha((x_1,y_1),(x_2,y_2))f(x_2,y_2)dx_2dy_2}, \quad f\in L^2((0,\infty)^m\times\R^n),
\end{equation}
and
$$0\leq \W_t^{\alpha}((x_1,y_1),(x_2,y_2))
    \leq C \frac{e^{-c\rho((x_1,y_1),(x_2,y_2))^2/t}}{|B_\rho((x_2,y_2),\sqrt{t})|}, \quad  (x_j,y_j)\in(0,\infty)^m\times\R^n,\;j=1,2,\mbox{and}\;t >0.$$

By defining $e^{-tG_{\alpha}}$, $t >0$, by \eqref{eqB.3} on $L^p((0,\infty)^m\times\R^n)$, $\{e^{-tG_{\alpha}}\}_{t>0}$
is a bounded semigroup on $L^p((0,\infty)^m\times\R^n)$, for every $1\leq p<\infty$.

Moreover, by \cite[Proposition 1.4]{Ouh} the semigroup $\{e^{-tG_{\alpha}}\}_{t>0}$ is bounded and holomorphic in
$L^2((0,\infty)^m\times\R^n)$ with angle $\pi/2$. According to \cite[Theorem 2.4]{Ouh}, for every
$(x_j,y_j)\in(0,\infty)^m\times\R^n,\;j=1,2$, the integral kernel $\W_t^{\alpha}((x_1,y_1),(x_2,y_2))$, $t>0$,
can be extended to an holomorphic function $\W_z^{\alpha}((x_1,y_1),(x_2,y_2))$, $Re(z)>0$.

By proceeding as in the proof of \cite[Lemmas 2.1 and 4.1]{DOS} and \cite[Lemma 3.3]{Sik}, for instance, we can show the following
properties of the integral heat kernel $\W_z^{\alpha}$, $Re(z)>0$.

\begin{Lem}\label{Lem2.1}
    Let $\alpha\in[1/2,\infty)^m$. Then,
    \begin{enumerate}
    \item[(a)] For every $(x_1,y_1)\in(0,\infty)^m\times\R^n$ and $t ,r>0$,
    $$\int_{((0,\infty)^m\times\R^n)\setminus B_\rho((x_1,y_1),r)}{|\W_t^{\alpha}((x_1,y_1),(x_2,y_2))|^2dx_2dy_2}
        \leq C \frac{e^{-r^2/ t}}{|B_\rho((x_1,y_1),\sqrt{t})|}.$$
    and
    \begin{align*}
    \|\W_t^{\alpha}((x_1,y_1),\cdot)\|^2_{L^2((0,\infty)^m\times\R^n)}= & \|\W_t^{\alpha}(\cdot,(x_1,y_1))\|^2_{L^2((0,\infty)^m\times\R^n)}
     \leq {C\over{|B_\rho((x_1,y_1),\sqrt{t})|}}.
    \end{align*}
    \item[(b)] For every $s>0$, $\mu\in\R$, $R>0$ and $(x_1,y_1)\in(0,\infty)^m\times\R^n$,
    \begin{align*}
    & \int_{(0,\infty)^m\times\R^n}{|\W_{(1+i\mu)R^{-2}}^{\alpha}((x_1,y_1),(x_2,y_2))|^2[\rho((x_1,y_1),(x_2,y_2))]^sdx_2dy_2}
    \leq C{{R^{-s}(1+|\mu|)^s}\over{|B_\rho((x_1,y_1),{1/ R})}|}.
    \end{align*}
    \end{enumerate}
\end{Lem}

Let $R>0$. As it was commented above the operator $e^{-R^{-2}G_\alpha}$ is bounded from
$L^1((0,\infty)^m\times\R^n)$ into $L^2((0,\infty)^m\times\R^n)$.
Suppose that $F$ is a bounded measurable function on $\R$ such that $supp\;F\subset[0,R^2]$. We define $H_2(\lambda)=e^{-\lambda R^{-2}}$, $\lambda\in\R$,
and $H_1={F/{H_2}}$. It is clear that $H_2(G_{\alpha})= e^{-R^{-2}G_\alpha}$
and that $H_1(G_{\alpha})$ is a bounded operator from
$L^2((0,\infty)^m\times\R^n)$ into itself with
$$\|H_1(G_{\alpha})\|_{L^2((0,\infty)^m\times\R^n)\rightarrow L^2((0,\infty)^m\times\R^n)}
    \leq\|H_1\|_{L^{\infty}(0,\infty)}\leq e \|F\|_{L^{\infty}(0,\infty)}.$$
Hence, the operator $F(G_{\alpha})$ is associated to the kernel
$$K_{F(G_{\alpha})}((x_1,y_1),(x_2,y_2))
    =H_1(G_{\alpha})[ \W_{R^{-2}}^{\alpha}(\cdot,(x_2,y_2))](x_1,y_1),\;\;(x_j,y_j)\in(0,\infty)^m\times\R^n,\;\;j=1,2.$$
Then, Lemma \ref{Lem2.1}, (a), leads to
$$\|K_{F(G_{\alpha})}((x_1,y_1),\cdot)\|^2_{L^2((0,\infty)^m\times\R^n)}
    \leq C \frac{\|F\|_{L^{\infty}(0,\infty)}^2}{|B_\rho((x_1,y_1),1/R)|}.$$

The arguments presented in the proof of \cite[Lemma 3.5]{Sik} (see also \cite[Lemma 4.3, (a)]{DOS}) allow us to establish the
following result.

\begin{Lem}\label{Lem2.2}
    Let $\alpha \in (-1/2,\infty)^m$ and $R,s>0$. For every $\varepsilon >0$ there exists $C_{\varepsilon} >0$ such that
    \begin{align*}
        & \int_{(0,\infty)^m\times\R^n}{|K_{F(G_{\alpha})}((x_1,y_1),(x_2,y_2))|^2[1+R\rho((x_1,y_1),(x_2,y_2))]^sdx_2dy_2}
        \leq C_{\varepsilon} \frac{\|{\delta}_{R^2}F\|^2_{W^{\infty}_{{s/ 2}+\varepsilon}}}{|B_\rho((x_1,y_1),{1/ R})|},
    \end{align*}
    for every bounded measurable function $F$ such that $supp\;F\subset[0,R^2]$.
\end{Lem}

We now define, for $j=1,...,m$ and $l=1,...,n$, the operators
$$M_jf(x,y)
    =x_jf(x,y), \quad f\in D(M_j)=\{g\in L^2((0,\infty)^m\times\R^n),\;\;x_jg\in L^2((0,\infty)^m\times\R^n)\},$$

$$\mathcal{D}_lf
    ={\mathcal{F}}_2^{-1}(-iu_l\mathcal{F}_2(f)(u)), \quad f\in D(\mathcal{D}_l )=\{g\in L^2((0,\infty)^m\times\R^n),\;\;u_l\mathcal{F}_2(g)\in L^2((0,\infty)^m\times\R^n)\}. $$

As it was done in \cite{MaSi}, we denote by $|M|$ the operator of multiplication by $|x|$, and  $|\mathcal{D}|$ represents the operator $\displaystyle\left(-\sum_{j=1}^n{{\partial^2}\over{\partial y_j^2}}\right)^{1/2}$.

It is clear that $|M|$ is a positive and selfadjoint operator. Moreover, we have that
$$|\mathcal{D}|f={\mathcal{F}}_2^{-1}(|u|\mathcal{F}_2(f)(u)), \;\;\;f\in D(|\mathcal{D}|).$$
Then,  $|\mathcal{D}|^df={\mathcal{F}}_2^{-1}(|u|^d\mathcal{F}_2(f)(u))$, $f\in D(|\mathcal{D}|^d)$, and $d\in\N$.
Also we define the operator $\mathcal{S}$ by
$$\mathcal{S}f(x,y)
    = |x|\mathcal{F}_2^{-1}\Big( |u| \mathcal{F}_2\Big(f(x,\cdot) \Big)(u) \Big)(y),$$
for every $f \in D(\mathcal{S})=\{f \in L^2((0,\infty)^m\times \mathbb{R}^n) : |x||u| \mathcal{F}_2(f(x,\cdot))(u) \in L^2((0,\infty)^m \times \mathbb{R}^n)\}$.
Note that $D(|M||\mathcal{D}|) \subsetneq D(\mathcal{S})$.

Next result is a version of \cite[Proposition 4]{MaSi} in our setting.

\begin{Lem}\label{Lem p.5}
    Let $\alpha\in (-1/2,\infty)^m$ and $\gamma>0$. Then,
    \begin{equation}\label{eqB.5.1}
        \||M|^{\gamma}f\|_{L^2((0,\infty)^m\times \R^n)}
            \leq C \|G_\alpha^{\gamma/2}|\mathcal{D}|^{-\gamma} f\|_{L^2((0,\infty)^m\times \R^n)},
    \end{equation}
    for every $f \in \textit{Ran}(|\mathcal{D}|^\gamma)$, the range of $|\mathcal{D}|^\gamma$, such that
    $|\mathcal{D}|^{-\gamma} f \in D(G_\alpha^{\gamma/2})$.
\end{Lem}

\begin{proof}
    Our first objective is to show that, for every $d\in\N$,
    \begin{equation}\label{eqB.4}
        \|\mathcal{S}^{2d}f\|_{L^2((0,\infty)^m\times\R^n)}
            \leq C\|G_{\alpha}^df\|_{L^2((0,\infty)^m\times\R^n)}, \quad f\in D(G_{\alpha}^d).
    \end{equation}
    We consider the operator
    $${\mathcal{A}}_{\alpha}(f)(x,y)
        ={\mathcal{F}}_2^{-1}\left[\sum_{k\in\N^m}{{{c_k^{\alpha}({\mathcal{F}}_2(f);u)}\over{2(2\frak{s}(k)+\frak{s}(\alpha) +m)|u|}}\Phi_k^{\alpha}(x;|u|)}\right](y), \quad  f \in D({\mathcal{A}}_{\alpha}),$$
    being
    $$D({\mathcal{A}}_{\alpha})=\left\{g\in L^2((0,\infty)^m\times\R^n)\;:\;\sum_{k\in\N^m}{{{c_k^{\alpha}({\mathcal{F}}_2(f);u)}\over{2(2\frak{s}(k)+\frak{s}(\alpha) +m)|u|}}\Phi_k^{\alpha}(x;|u|)}\in L^2((0,\infty)^m\times\R^n)\right\}.$$
    Observe that, for every $f\in D({\mathcal{A}}_{\alpha})$, ${\mathcal{A}}_{\alpha}f\in D(G_{\alpha})$ and $G_{\alpha}{\mathcal{A}}_{\alpha}f=f$.
    Furthermore, if $f\in D(G_{\alpha})$, $G_{\alpha}f \in  D({\mathcal{A}}_{\alpha})$ and ${\mathcal{A}}_{\alpha}G_{\alpha}f=f$.

    We treat the case $d=1$. The inequality in \eqref{eqB.4} is now equivalent to the following one
    \begin{equation}\label{eqB.5}
        \|\mathcal{S}^{2}{\mathcal{A}}_{\alpha}f\|_{L^2((0,\infty)^m\times\R^n)}
            \leq C\|f\|_{L^2((0,\infty)^m\times\R^n)}, \quad f\in D({\mathcal{A}}_{\alpha}).
    \end{equation}
    According to Plancherel equality for the Fourier transform, \eqref{eqB.5} holds if, and only if,
    $D(G_\alpha) \subseteq D(\mathcal{S}^2)=\{ f \in L^2((0,\infty)^m\times \mathbb{R}^n) : |x|^2|u|^2 \mathcal{F}_2(f(x,\cdot))(u) \in L^2((0,\infty)^m \times \mathbb{R}^n) \}$ and
    the operator $T_{\alpha}$ defined by
    $$T_{\alpha}(f)(x,u)
        =\sum_{k\in\N^m}{{{|x|^2|u|c_k^{\alpha}({\mathcal{F}}_2(f);u)}\over{2(2\frak{s}(k)+\frak{s}(\alpha) +m)}}\Phi_k^{\alpha}(x;|u|)},\;\;f\in L^2((0,\infty)^m\times\R^n)$$
    is bounded from $L^2((0,\infty)^m\times\R^n)$ into itself.

    In order to show the $L^2$-boundedness property for $T_{\alpha}$ we consider the operator $\frak{L}_{\alpha}^{-1}$ defined by
    $$\frak{L}_{\alpha}^{-1}g
        =\sum_{k\in\N^m}{{{c_k^\alpha(g)}\over{2(2\frak{s}(k)+\frak{s}(\alpha) +m)}}\Phi_k^{\alpha}}, \quad g\in L^2((0,\infty)^m),$$
    where
    \begin{equation}\label{ck}
        c_k^\alpha(g)=\int_{(0,\infty)^m}{{\varphi}_k^{\alpha}(x)g(x)dx}, \quad k\in\N^m.
    \end{equation}

    The operator $|x|^2\frak{L}_\alpha^{-1}$ is bounded from $L^2((0,\infty)^m)$ into itself. Indeed, let $g \in L^2((0,\infty)^m)$.
    We define the function $g_0$ by
    $$g_0(x)=g(x), \ x \in (0,\infty)^m, \quad \text{and} \quad g_0(x)=0, \ x \in\R^m\setminus\{(0,\infty)^m\}.$$
    We have that $\frak{L}_\alpha^{-1}(|g|)(x) \leq C H^{-1}(|g_0|)$, $x \in (0,\infty)^m$, where $H$ represents the Hermite operator.
    Then, the $L^2$-boundedness of the operator $|x|^2\frak{L}_\alpha^{-1}$ follows from \cite[Lemma 3]{BT} (see also \cite{BT2}).
    We get
    \begin{align*}
        \|T_{\alpha}(f)\|^2_{L^2((0,\infty)^m\times\R^n)}
        & = \int_{\R^n}\int_{(0,\infty)^m}{\left|\sum_{k\in\N^m}{{{|x|^2|u|c_k^{\alpha}({\mathcal{F}}_2(f);u)}\over{2(2\frak{s}(k)+\frak{s}(\alpha) +m)}}\Phi_k^{\alpha}(x;|u|)}\right|^2dxdu} \\
        & = \int_{\R^n}\int_{(0,\infty)^m}{\left|\sum_{k\in\N^m}{{{|x|^2c_k^\alpha\Big({\mathcal{F}}_2(f)\Big(y/\sqrt{|u|},u\Big)\Big)}\over{2(2\frak{s}(k)+\frak{s}(\alpha) +m)}}\Phi_k^{\alpha}(x)}\right|^2dx{{du}\over{|u|^{m/2}}}} \\
        & = \int_{\R^n}\int_{(0,\infty)^m}{\left||x|^2\frak{L}_{\alpha}^{-1}\Big({\mathcal{F}}_2(f)\Big({y\over{\sqrt{|u|}}},u\Big)\Big)(x)\right|^2dx{{du}\over{|u|^{m/2}}}} \\
        & \leq C  \int_{\R^n}\int_{(0,\infty)^m}{\Big|{\mathcal{F}}_2(f)\Big({y\over{\sqrt{|u|}}},u\Big)\Big|^2dy{{du}\over{|u|^{m/2}}}} \\
        & \leq C\|f\|^2_{L^2((0,\infty)^m\times\R^n)}, \quad f\in L^2((0,\infty)^m\times\R^n).
    \end{align*}
    Note that $\mathcal{F}_2(f)(\cdot/\sqrt{|u|},u) \in L^2((0,\infty)^m)$, a.e. $u \in \mathbb{R}^n$,
    and then the coefficient $c_k^\alpha$ in the second equality above, which is given by \eqref{ck}, is understood as a function of $u$.
    On the other hand, the property $D(G_\alpha) \subseteq D(\mathcal{S}^2)$
    can be also  deduced from the previous argument.

    An inductive procedure allows to show that \eqref{eqB.4} is true for every $d \in \N$.
    The imaginary powers of the operators $\mathcal{S}$ and $G_\alpha$ are bounded in $L^2((0,\infty)^m\times \R^n)$ (see \cite[p. 640]{Me1} or \cite[Theorem B]{Ya}).
    By using \cite[Theorem 11.6.1]{MaCa} and \cite[Theorem 4.1.2]{BL} we get that there exists $C>0$
    such that
    $$ \|\mathcal{S}^{\gamma}f\|_{L^2((0,\infty)^m\times \R^n)}
        \leq C \Big( \|f\|_{L^2((0,\infty)^m\times \R^n)} + \|G_\alpha^{\gamma/2} f\|_{L^2((0,\infty)^m\times \R^n)}\Big), \quad  f \in D(G_\alpha^{\gamma/2}).$$
    In the usual way, the homogeneity allows us to obtain that
    $$ \|\mathcal{S}^{\gamma}f\|_{L^2((0,\infty)^m\times \R^n)}
        \leq C \|G_\alpha^{\gamma/2} f\|_{L^2((0,\infty)^m\times \R^n)}, \quad  f \in D(G_\alpha^{\gamma/2}).$$

    Since $|\mathcal{D}|^\gamma$ is an one to one operator, we deduce that \eqref{eqB.5.1} holds.
\end{proof}

\begin{Lem}\label{Lem p.6}
    Suppose that $H$ is a compactly supported Borel measurable complex function defined on $\R$.
    For every $f \in C_c^\infty((0,\infty)^m) \otimes C_c^\infty(\R^n)$, we have that
    $$H(G_\alpha)f(x,t)
        = \int_{\R^n} \int_{(0,\infty)^m} K_H^\alpha(y,z;x,t)f(y,z)dydz, \quad  x \in (0,\infty)^m \text{ and } t \in \R^n,$$
    being, for $ x,y \in (0,\infty)^m \text{ and } z,t \in \R^n$,
    $$K_H^\alpha(y,z;x,t)
        = \frac{1}{(2\pi)^n} \int_{\R^n} \sum_{k\in\N^m} H\Big(2(2\frak{s}(k)+\frak{s}(\alpha)+m)|u|\Big) \Phi_k^\alpha(x;|u|) \Phi_k^\alpha(y;|u|) e^{-iu\cdot(z-t)}du.$$
    Moreover, for $y \in (0,\infty)^m \text{ and } z \in \R^n$,
    \begin{align}\label{7.1}
        \| K_H^\alpha(y,z;\cdot ,\cdot) \|_{L^2((0,\infty)^m\times \R^n)}^2
            &= \int_{\R^n} \sum_{k\in\N^m} \Big|  H\Big(2(2\frak{s}(k)+\frak{s}(\alpha)+m)|u|\Big) \Phi_k^\alpha(y;|u|) \Big|^2  du.
    \end{align}
\end{Lem}

\begin{proof}
    We consider $f(x,t)=h(x)g(t)$, where $h \in C_c^\infty((0,\infty)^m)$ and $g \in C_c^{\infty}(\R^n)$. We can write
    \begin{align*}\label{eqB.6}
        & H(G_\alpha)f(x,t)
            = \mathcal{F}_2^{-1}\Big[ \sum_{k\in\N^m} H\Big(2(2\frak{s}(k)+\frak{s}(\alpha)+m)|u|\Big) \nonumber\\
        & \qquad \qquad\times \int_{(0,\infty)^m} \Phi_k^\alpha(y;|u|)h(y)dy \mathcal{F}_2(g)(u) \Phi_k^\alpha(x;|u|) \Big](t) \nonumber \\
        & \qquad = \frac{1}{(2\pi)^n} \int_{\R^n} e^{iu\cdot t} \int_{(0,\infty)^m}  \sum_{k\in\N^m} H\Big(2(2\frak{s}(k)+\frak{s}(\alpha)+m)|u|\Big) \Phi_k^\alpha(x;|u|) \Phi_k^\alpha(y;|u|) h(y)dy \nonumber \\
        & \qquad \qquad \times \int_{\R^n} e^{-iu\cdot z}g(z)dzdu \\
        & \qquad = \int_{\R^n} \int_{(0,\infty)^m} K_H^\alpha(y,z;x,t)h(y)g(z)dydz, \quad  x \in (0,\infty)^m \text{ and } t \in \R^n.
    \end{align*}
    Indeed, the interchange of the order of integration can be justified as follows.
    Since $H$ has bounded support there exists $b>0$ such that $H(2(2\frak{s}(k)+\frak{s}(\alpha) +m)|u|)=0$, provided that $2(2\frak{s}(k)+\frak{s}(\alpha) +m)|u|>b$, $u\in\R^n$ and $k\in\N^m$.
    Then, $H(2(2\frak{s}(k)+\frak{s}(\alpha) +m)|u|)=0$, when $2(\frak{s}(\alpha) +m)|u|>b$, $u\in\R^n$ and $k\in\N^m$. Hence, we have that
    \begin{align*}
        & \int_{\R^n}  \int_{(0,\infty)^m} \int_{\R^n} |g(z)||h(y)|\sum_{k\in\N^m}|H\Big(2(2\frak{s}(k)+\frak{s}(\alpha)+m\Big)|u|)||\Phi_k^\alpha(x;|u|)||\Phi_k^\alpha(y;|u|)|dzdydu \\
        & \qquad \leq C \int_{B(0,b/(2\frak{s}(\alpha)+2m))} \int_{(0,\infty)^m}\int_{\R^n}|g(z)||h(y)|\\
        & \qquad \qquad \times \sum_{k\in\N^m,\;2\frak{s}(k)+\frak{s}(\alpha) +m<b/(2|u|)}|\Phi_k^\alpha(x;|u|)||\Phi_k^\alpha(y;|u|)|dzdydu
    \end{align*}
    Furthermore, since $|\Phi_k^\alpha(z)|\leq C$, $k\in\N^m$ and $z\in (0,\infty)^m$, for a certain $C>0$ (see \cite[(27)]{No}, we get
    \begin{align*}
        & \int_{B(0,b/(2\frak{s}(\alpha)+2m))} \int_{(0,\infty)^m}\int_{\R^n}|g(z)||h(y)|\sum_{k\in\N^m,\;2\frak{s}(k)+\frak{s}(\alpha) +m<b/(2|u|)}|\Phi_k^\alpha(x;|u|)||\Phi_k^\alpha(y;|u|)|dzdydu \\
        & \qquad \leq C \|g\|_{L^1(\R^n)} \int_{(0,\infty)^m}|h(y)| \int_{B(0,b/(2\frak{s}(\alpha)+2m))} \sum_{k\in\N^m,\;2\frak{s}(k)+\frak{s}(\alpha) +m<b/(2|u|)}|u|^{{m\over 2}}dydu \\
        & \qquad \leq C \|g\|_{L^1(\R^n)} \int_{(0,\infty)^m}|h(y)| \int_{B(0,b/(2\frak{s}(\alpha)+2m))} {{|u|^{{m\over 2}}}\over{|u|^m}}dydu \\
        & \qquad \leq C \|g\|_{L^1(\R^n)} \|h\|_{L^1((0,\infty)^m)} < \infty.
    \end{align*}
    We have made the change of variable $|u|^{1/2}y=Y$ and have taken into account that $\sharp\{k\in\N^m:\;\frak{s}(k)\leq \ell\}\leq C \ell^m$, where $\sharp A$ represents the cardinal of $A$.

    On the other hand, by using Plancherel equality for Fourier transforms and Laguerre expansions we easily obtain \eqref{7.1}.
\end{proof}

For every $R>0$, we define the weight function,
$$w_R((x,t),(y,z))=\min \{R,1/|y|\}|x|, \quad  x,y \in (0,\infty)^m \text{ and } t,z \in \R^n.$$
The following is our crucial weighted Plancherel inequality.

\begin{Lem}\label{Lem p.7}
    Assume that $\gamma \in [0,n/2)$ and $H$ is a compactly supported Borel measurable complex function defined on $\R$.
    Then, for $y \in (0,\infty)^m \text{ and } z \in \R^n$,
    \begin{equation}\label{eqB.9}
        \| |M|^\gamma K_{H}^\alpha (y,z;\cdot,\cdot) \|_{L^2((0,\infty)^m\times \R^n)}^2
            \leq C \int_0^\infty |H(\omega)|^2 {\omega}^{(n+m)/2}\min\{{\omega}^{-\gamma+n/2},|y|^{2\gamma-n}\} {{d{\omega}}\over {\omega}}.
    \end{equation}
    Particulary, when $\supp H \subset [R^2,4R^2]$, for some $R>0$, we have that
    \begin{align}\label{8.1}
        & \sup_{(y,z) \in (0,\infty)^m \times \R^n} |B_\rho((y,z),1/R)|^{1/2} \| w_R((\cdot,\cdot),(y,z))^\gamma K_H ^\alpha(y,z;\cdot,\cdot)\|_{L^2((0,\infty)^m \times \R^n)}
        \leq C \|\delta_{R^2}H\|_{L^2(\R)},
    \end{align}
    being $C>0$ independent on $R$.
\end{Lem}

\begin{proof}
    We define, for every $\ell \in \N$,
    $$H_{\ell}(\omega)
        = \chi_{(1/{\ell},\ell)}(\omega)H(\omega), \quad  \omega \in \R.$$
    By using monotone convergence theorem it follows that
    $$\| K_H^\alpha(y,z;\cdot ,\cdot) \|_{L^2((0,\infty)^m\times \R^n)}
        = \lim_{\ell \to \infty} \| K_{H_{\ell}}^\alpha(y,z;\cdot ,\cdot) \|_{L^2((0,\infty)^m\times \R^n)}, \quad  y \in (0,\infty)^m \text{ and } z \in \R^n.$$
    Let $\ell \in \N$.
    Note that
    $$\int_{\R^n} \sum_{k\in\N^m} \Big|  H_{\ell}\Big(2(2\frak{s}(k)+\frak{s}(\alpha)+m)|u|\Big) \Phi_k^\alpha(y;|u|) \Big|^2  \frac{\Big(2(2\frak{s}(k)+\frak{s}(\alpha)+m)|u|\Big)^\gamma}{|u|^{2\gamma}} du < \infty$$
    if, and only if,
    $$\int_{\R^n} \sum_{k\in\N^m} \Big|  H_{\ell}\Big(2(2\frak{s}(k)+\frak{s}(\alpha)+m)|u|\Big) \Phi_k^\alpha(y;|u|) \Big|^2  \frac{du}{|u|^{2\gamma}}  < \infty.$$

    Our next objective is to estimate the following function, for each $y \in (0,\infty)^m$
    $$\Lambda_{\ell}(y)
        = \int_{\R^n} \sum_{k\in\N^m} \Big|  H_{\ell}\Big(2(2\frak{s}(k)+\frak{s}(\alpha)+m)|u|\Big) \Phi_k^\alpha(y;|u|) \Big|^2  \frac{\Big(2(2\frak{s}(k)+\frak{s}(\alpha)+m)|u|\Big)^\gamma}{|u|^{2\gamma}} du. $$
    By making straightforward manipulations we get, for each $y \in (0,\infty)^m$
    \begin{align*}
        \Lambda_{\ell}(y)
            & \leq C \int_0^\infty |H_{\ell}(\omega)|^2\sum_{k\in\N^m} \Big| \Phi_k^\alpha\Big(\frac{\sqrt{\omega}y}{\sqrt{2(2\frak{s}(k)+\frak{s}(\alpha)+m)}}\Big) \Big|^2  \frac{{\omega}^{n+m/2-1-\gamma}}{(2\frak{s}(k)+\frak{s}(\alpha)+m)^{n+m/2-2\gamma}} d{\omega}.
    \end{align*}
    According to {\cite[p. 1124]{MW}}, there exist $C$, $\eta$, $\lambda$ and $\xi \in (0,\infty)$
    such that
    $$|\varphi_k^\alpha(x)|
        \leq C \mathcal{M}_k^\alpha(x), \quad  x \in (0,\infty), \;\alpha>-1/2,  \text{ and } k \in \N,$$
    where
    $$\mathcal{M}_k^\alpha(x)
        = x^{\alpha+1/2} \Big( \frac{1}{\nu_k} + x^2 \Big)^{-1/4-\alpha/2} (\nu_k^{1/3} + |x^2-\nu_k|)^{-1/4} \Psi_k^\alpha(x), \quad  x \in (0,\infty),$$
    and
    $$\Psi_k^\alpha(x)
        = \left\{ \begin{array}{ll}
                    1, & 0 \leq x^2 \leq \nu_k, \\
                    \exp\Big( -\eta |\nu_k-x^2|^{3/2}/\nu_k^{1/2}\Big), & \nu_k \leq x^2 \leq (1+\lambda)\nu_k, \\
                    e^{-\xi x^2}, & (1+\lambda)\nu_k \leq x^2,
                \end{array}\right.$$
    being $\nu_k = 4k +2 \alpha +2$, $k \in \N$.

    Then, we deduce that
    \begin{equation}\label{eqB.7}
        |\varphi_k^\alpha(x)|
        \leq C \left\{ \begin{array}{ll}
                        (\nu_k^{1/3} + |x^2-\nu_k|)^{-1/4}, & x \in (0,\infty), \\
                        e^{-\xi x^2}, & x^2 \geq (1+\lambda)\nu_k.
                \end{array}\right.
    \end{equation}

    Furthermore, as it was done in \cite[Lemma 8]{MaSi}, one can get for all $k\in\N^m$ and $\alpha\in (-1/2,\infty)^m$
     \begin{equation}\label{eqB.7.1}
        |\Phi_k^\alpha(x)|
        \leq C \left\{ \begin{array}{ll}
                        \nu_k^{m/2-1}, & x \in (0,\infty)^m, \\
                        e^{-\xi |x|^2}, & |x|^2 \geq (1+\lambda)\nu_k.
                \end{array}\right.
    \end{equation}
    where, in this case, $\nu_k=2(2\frak{s}(k)+\frak{s}(\alpha)+m)$.
    By proceeding as in the proof of \cite[Lemma 9]{MaSi}, \eqref{eqB.7} and  \eqref{eqB.7.1} allows us to obtain that, for every
    $\varepsilon>0$,
    \begin{equation}\label{eqB.8}
        \sup_{x \in (0,\infty)^m} \sum_{k\in\N^m}\frac{\max\{1,|x|\}^\varepsilon}{(2\frak{s}(k)+\frak{s}(\alpha)+m)^{\varepsilon +m/2}}
                \Big| \varphi_k^\alpha \Big( \frac{x}{\sqrt{2(2\frak{s}(k)+\frak{s}(\alpha)+m)}} \Big) \Big|^2 < \infty.
    \end{equation}
    From \eqref{eqB.8} we deduce that
    $$\Lambda_{\ell}(y)
        \leq C \int_0^\infty |H_{\ell}(\omega)|^2 {\omega}^{(n+m)/2}\min\{{\omega}^{-\gamma+n/2},|y|^{2\gamma-n}\} {{d\omega}\over{\omega}}, \quad  y \in (0,\infty)^m,$$
    provided that $\gamma \in [0,n/2)$.

    By Lemma~\ref{Lem p.5}, we get, for every $ y \in (0,\infty)^m \text{ and } z \in \R^n$,
    \begin{align*}
        \| |M|^\gamma K_{H_{\ell}}^\alpha (y,z;\cdot,\cdot) \|_{L^2((0,\infty)\times \R)}
            \leq C \int_0^\infty |H_{\ell}(\omega)|^2 {\omega}^{(n+m)/2}\min\{{\omega}^{-\gamma+n/2},|y|^{2\gamma-n}\} {{d\omega}\over{\omega}}.
    \end{align*}
    where $C$ does not depend on $\ell$. By taking limits as $\ell \to \infty$ we obtain \eqref{eqB.9}.

    Suppose now that $\supp H \subset [R^2,4R^2]$, where $R>0$. It is clear that
    \begin{align*}
        & \int_0^\infty |H(\omega)|^2 {\omega}^{(n+m)/2}\min\{{\omega}^{-\gamma+n/2},|y|^{2\gamma-n}\} {{d\omega}\over{\omega}} \\
         & \qquad \qquad  = \int_1^4 |H(R^2v)|^2 R^{n+m}v^{(n+m)/2-1}\min\{R^{-2\gamma+n}v^{-\gamma+n/2},|y|^{2\gamma-n}\} dv \\
        & \qquad \qquad \leq C R^{n+m}  \min\{R^{n-2\gamma},|y|^{2\gamma-n}\}  \int_1^4 |H(R^2v)|^2v^{(n+m)/2-1} dv, \quad  y \in (0,\infty)^m.
    \end{align*}
    Since $|B_\rho((x,t),R)| \sim R^{m+n} \max\{|x|,R\}^n$, $x\in (0,\infty)^m$, $R>0$ and $t \in \R^n$ (\cite[Proposition 3, (9)]{MaSi}), we deduce \eqref{8.1}.
\end{proof}

The proof of Theorem~\ref{Th1.1} can be finished now by proceeding as in \cite[Section 4]{MaSi} (see also \cite[Section 4]{CS}).

\section{Proof of Theorem \ref{Th1.2}} \label{sec:3}

\subsection{Some definitions and estimates}

Let $\beta\geq 1/2$. The Laguerre operator $-L_{\beta}$ generates the semigroup of contractions $\{W_t^{\beta}\}_{t>0}$ in $L^2(0,\infty)$, where, for every $t>0$,
$$W_t^{\beta}(g)
    =\sum_{k=0}^{\infty}{e^{-2t(2k+\beta +1)}c_k^\beta(g)\varphi_k^{\beta}}, \quad g\in L^2(0,\infty),$$
where, for every $k\in\N$,
$$c_k^\beta(g)
    =\int_0^{\infty}{\varphi_k^{\beta}(y)g(y)dy}.$$
According to the Mehler's formula \cite[(1.1.47)]{Th}, for every $t>0$, we can write
\begin{equation}\label{eq2.1}
    W_t^{\beta}(g)(x)=\int_0^{\infty}{ W_t^{\beta}(x,y)g(y)dy},\;\;\;x\in(0,\infty),
\end{equation}
for every $g\in L^2(0,\infty)$, where, for every $t,x,y\in (0,\infty)$
$$W_t^{\beta}(x,y)
    =\left({{2e^{-2t}}\over{1-e^{-4t}}}\right)^{1/2}\left({{2xye^{-2t}}\over{1-e^{-4t}}}\right)^{1/2}
        I_{\beta}\left({{2xye^{-2t}}\over{1-e^{-4t}}}\right)\exp\left(-\frac{1}{2}(x^2+y^2){{1+e^{-4t}}\over{1-e^{-4t}}}\right).$$
To simplify notation, it is convenient to consider the functions
    \begin{equation}\label{fafb}
        \fa = \frac{2e^{-2t}}{1-e^{-4t}} \quad  \text{ and } \quad
        \fb = \frac{1}{2}\frac{1+e^{-4t}}{1-e^{-4t}}, \quad  t>0.
    \end{equation}
    Observe that
    \begin{equation}\label{ab}
        \fa t \leq 1/2, \qquad \fb t \geq 1/4 \quad \text{and} \quad  |1-2 \fb| \leq \fa, \qquad t \in (0,\infty),
    \end{equation}
    and these bounds will be used repeatedly.

Here $I_{\beta}$ denotes the modified Bessel function of the first
kind and order $\beta$. By defining $W_t^{\beta}$, $t>0$, in
$L^p(0,\infty)$, $1\leq p\leq\infty$, by the integral in
\eqref{eq2.1}, $\{W_t^{\beta}\}_{t>0}$ is a semigroup  of
contractions in $L^p(0,\infty)$, $1\leq p\leq\infty$ (see
\cite[Theorem 4.1]{NoSt3}).

In the sequel we will use the following properties of the Bessel function $I_{\nu}$, $\nu>-1/2$.
By \cite[(5.11.8) and (5.16.4)]{Leb}, for every $n\in\N $,
\begin{equation}\label{eq2.8}
    \sqrt{2\pi z}I_{\nu}(z)e^{-z}
        =\sum_{r=0}^n (-1)^r\frac{[\nu ,r]}{(2z)^{r}}+\mathcal{O} \Big(\frac{1}{z^{n+1}}\Big),\;\;z\in (0,\infty),
\end{equation}
where $[\nu ,0]=1$ and
$$[\nu ,r]
    ={{(4{\nu}^2-1)(4{\nu}^2-3^2) \cdots (4{\nu}^2-(2r-1)^2)}\over{2^{2r}\Gamma(r+1)}}, \quad r=1,2, ...$$
Also, it is clear that
        \begin{equation}\label{eq2.7.1}
    I_{\nu}(z)
        \sim {z^{\nu}\over{2^{\nu}\Gamma(\nu +1)}},\;\;\mbox{as}\; z\rightarrow 0^+,
\end{equation}
Moreover, according to \cite[(5.7.9)]{Leb} we have that,
\begin{equation}\label{eq2.7}
    {d\over{dz}}(z^{-\nu}I_{\nu}(z))
        =z^{-\nu}I_{\nu +1}(z), \;\;\;z\in(0,\infty),
\end{equation}
and
\begin{equation}\label{22b}
    I_{\nu+1}(z)
        = I_{\nu-1}(z)-\frac{2\nu}{z} I_\nu(z), \;\;\;z\in(0,\infty).
\end{equation}

Let now $\alpha=(\alpha_1,...,\alpha_m)\in[1/2,\infty)^m$. The Laguerre operator $-\frak{L}_{\alpha}$ generates the semigroup of contractions  $\{\mathcal{W}_t^{\alpha}\}_{t>0}$ in $L^2((0,\infty)^m)$, where, for every $t>0$ and $g\in L^2((0,\infty)^m)$,
$$\mathcal{W}_t^{\alpha}(g)(x)=\int_{(0,\infty)^m}{\mathcal{W}_t^{\alpha}(x,y)g(y)dy},\;\;\;x\in (0,\infty)^m,$$
being
$$\mathcal{W}_t^{\alpha}(x,y)=\prod_{j=1}^m{W_t^{\alpha_j}(x_j,y_j)},\;\;t>0,\;\;(x_1,...,x_m)\in (0,\infty)^m\;\mbox{and}\;y=(y_1,...,y_m)\in (0,\infty)^m.$$

The family $\{\mathcal{W}_t^{\alpha}\}_{t>0}$ is a semigroup of contractions in $L^p((0,\infty)^m)$, $1\leq p\leq\infty$.

Assume that $a>0$. Straightforward manipulations allow us to show that the operator $-\frak{L}_{\alpha}(a)$ generates
on $L^p(0,\infty)$, $1\leq p<\infty$, the semigroup of operators $\{\mathcal{W}_{t,a}^{\alpha}\}_{t>0}$, where, for every $t>0$,
$$\mathcal{W}_{t,a}^{\alpha}(g)(x)
    =\int_{(0,\infty)^m}{ \mathcal{W}_t^{\alpha}(x,y;a)g(y)dy}, \quad g\in L^p((0,\infty)^m),\;1\leq p<\infty,$$
being
$$\mathcal{W}_t^{\alpha}(x,y;a)
    =a^{m/2}\mathcal{W}_{ta}^{\alpha}(\sqrt{a}x,\sqrt{a}y), \quad x,y\in (0,\infty)^m\;\mbox{and}\;t>0.$$
We can write, for every $g\in L^p((0,\infty)^m),\;1< p<\infty$,
$${\frak{L}_{\alpha}(a)}^{-{1/2}}g(x)
    ={1\over{\sqrt{\pi}}}\int_0^{\infty}{\mathcal{W}_{t,a}^{\alpha}(g)(x){{dt}\over{\sqrt{t}}}}, \quad x\in (0,\infty)^m.$$
By using the arguments given in \cite{BFRS}, \cite{FGS} and \cite{MS} we can see that, for every  $j=1,...,m$, and $g\in L^2((0,\infty)^m)$,

\begin{equation}\label{eq2.2}
    R_{\alpha, j}(a)(g)(x)
        =\lim_{\varepsilon\rightarrow 0^+}\int_{|x-y|>\varepsilon}{ R_{\alpha, j}(x,y;a)g(y)dy}, \quad \mbox{a.e.}\;x\in (0,\infty)^m,
\end{equation}
where
$$R_{\alpha, j}(x,y;a)
    ={1\over{\sqrt{\pi}}}\int_0^{\infty}{A_{\alpha_j}(a) \mathcal{W}_t^{\alpha}(x,y;a){{dt}\over{\sqrt{t}}}}, \quad x,y\in (0,\infty)^m,\;x\neq y,$$
    and
\begin{equation*}
    \widetilde {R}_{\alpha, j}(a)(g)(x)
        =\lim_{\varepsilon\rightarrow 0^+}\int_{|x-y|>\varepsilon}{\widetilde {R}_{\alpha, j}(x,y;a)g(y)dy}, \quad \mbox{a.e.}\;x\in (0,\infty)^m,
\end{equation*}
where
$$\widetilde {R}_{\alpha, j}(x,y;a)
    ={1\over{\sqrt{\pi}}}\int_0^{\infty}{A_{\alpha_j}^*(a) \mathcal{W}_t^{\alpha}(x,y;a){{dt}\over{\sqrt{t}}}}, \quad x,y\in (0,\infty)^m,\;x\neq y.$$

Moreover, the operators $R_{\alpha, j}(a)$ and $\widetilde {R}_{\alpha, j}(a)$, $j=1,...,m$, can be extended from $L^2((0,\infty)^m)\cap L^p((0,\infty)^m)$
to $L^p((0,\infty)^m)$ as a bounded operator from $L^p((0,\infty)^m)$ into itself, for every $1<p<\infty$.

We need to recall some definitions related to the Hermite operator which we denote by $\mathcal{H}$.

The operator
$$\mathcal{H}=-\Delta + |x|^2,\;\;\;\;\mbox{on}\;\R^m$$
can be written as follows
$$\mathcal{H}=-{1\over 2}\sum_{j=1}^m{(A_jA_j^*+A_j^*A_j)},$$
where
\begin{equation}\label{13.1}
    A_j=\partial_{x_j}+x_j
    \quad \text{and} \quad
    A_j^*=-\partial_{x_j}+x_j, \quad j=1,...,m.
\end{equation}
Note that $A_j^*$ is the "formal" adjoint of $A_j$ in $L^2(\R^n)$, $j=1,...,m$.
For every $k\in\N$, the $k$-th Hermite function $h_k$ is defined by
$$h_k(x)
    =(\sqrt{\pi}2^kk!)^{-{1/2}}e^{-{{x^2}/{2}}}H_k(x), \quad x\in\R,$$
 where $H_k$ represents the $k$-th Hermite polynomial \cite[(4.9.1) and (4.9.2)]{Leb}. We have that
$$Hh_k
    =(2k+1)h_k, \quad k\in\N.$$
where $\displaystyle H
    =-{{d^2}\over{dx^2}}+x^2, \quad x\in\R$. The system $\{h_k\}_{k\in\N}$ is an orthonormal basis in $L^2(\R)$.
The operator $-H$ generates the semigroup of contractions $\{W_t\}_{t>0}$ in $L^2(\R)$, being
for every $t>0$,
$$W_t(g)
    =\sum_{k=0}^{\infty}{e^{-t(2k+1)}b_k(g)h_k}, \quad g\in L^2(\R),$$
and
$$b_k(g)
    = \int_\R h_k(y)g(y)dy, \quad k \in \N.$$
By using Mehler's formula for Hermite functions \cite[(1.1.36)]{Th}  we obtain
\begin{equation}\label{eq2.5}
    W_t(g)(x)
        =\int_\R {W_t(x,y)g(y)dy}, \quad g\in L^2(\R)\;\mbox{and}\;t>0,
\end{equation}
where, for $x,y\in\R\;\mbox{and}\;t\in (0,\infty)$,
\begin{equation}\label{Hermitekernel}
    W_t(x,y)
        ={1\over{\sqrt{\pi}}}\left({{e^{-2t}}\over{1-e^{-4t}}}\right)^{1/2}\exp\left(-\frac{1}{4}\left[(x-y)^2{{1+e^{-2t}}\over{1-e^{-2t}}}+(x+y)^2{{1-e^{-2t}}\over{1+e^{-2t}}}\right]\right).
\end{equation}
Moreover, if $W_t$, $t>0$, is defined by \eqref{eq2.5}, $\{W_t\}_{t>0}$ is a
semigroup of contractions in $L^p(\R)$, $1\leq p<\infty$.

The Hermite operator $-\mathcal{H}$ generates, for every $1\leq p<\infty$, the semigroup of contractions  $\{\mathcal{W}_t\}_{t>0}$ in $L^p(\R^m)$, where, for every $f\in L^p(\R^m)$,
$$\mathcal{W}_t(f)(x)=\int_{\R^m}{\mathcal{W}_t(x,y)f(y)dy}, \;\;\;x\in\R^m,$$
being
$$\mathcal{W}_t(x,y)=\prod_{j=1}^m{W_t(x_j,y_j)},\;\;\;x,\;y\in\R^m\;\mbox{and}\;t>0.$$

In order to study Riesz transforms associated with Grushin operator Jotsaroop, Sanjay and Thangavelu \cite{JST}
considered the scaled Hermite operator $\mathcal{H}(a)$ defined by
$$\mathcal{H}(a)
    =-\Delta +a^2|x|^2, \quad \mbox{on}\;\R^m,$$
for every $a\in\R$. The operator $\mathcal{H}(a)$ can be written as
$$\mathcal{H}(a)
    = \frac{1}{2}\sum_{j=1}^m[A_j(a)A_j^*(a)+A_j^*(a)A_j(a)],$$
where $\displaystyle A_j(a)={d\over{dx_j}}+|a|x_j$ and $\displaystyle A_j^*(a)=-{d\over{dx_j}}+|a|x_j$, $j=1,...,m$.
Riesz transforms for the operator $\mathcal{H}(a)$ were formally defined by
$$R_j(a)
    =A_j(a){\mathcal{H}(a)}^{-{1/2}}\;\;\;\mbox{and}\;\;\;\widetilde{R}_j(a)=A_j^*(a){\mathcal{H}(a)}^{-{1/2}}, \quad  a \in \R \setminus \{0\},\;\mbox{and}\;j=1,...,m.$$
Here, we only consider the Riesz transform $R_1(a)$.  By taking in mind
\cite[(3.1)]{ST1} the operator $R_1(a)$ is defined in $L^2(\R^m)$ as follows
$$R_1(a)(g)
    =  \sum_{k\in\N^m} \sqrt{\frac{2k_1}{2\frak{s}(k)+1}} b_k(a)(g) \frak{h}_{k-1}(\cdot;a), \quad  g \in L^2(\R^m),$$
where, for every $k \in\N^m$,
$$b_k(a)(g)
    = \int_{\R^m} \frak{h}_k(y;a)g(y)dy,$$
and $\frak{h}_k(x;a)=|a|^{m/4} \frak{h}_k(\sqrt{|a|}x)$, $x \in \R^m$, being $\frak{h}_k(x)\displaystyle=\prod_{j=1}^m{h_{k_j}(x_j)}$, $x \in\R^m$.

 $R_1(a)$ is a bounded operator in
$L^2(\R^m)$. Moreover, for every $1<p<\infty$, $R_1(a)$ can be extended from $L^2(\R^m)\cap L^p(\R^m)$ to $L^p(\R^m)$
as a bounded operator from $L^p(\R^m)$ into itself, and, for every $g \in L^p(\R^m)$,
$$R_1(a)(g)(x)
    =\lim_{\varepsilon\rightarrow 0^+}\int_{|x-y|>\varepsilon}{ R_1(x,y;a)g(y)dy}, \quad \mbox{a.e.}\;x\in \R^m,$$
where
\begin{align*}
    R_1(x,y;a)
    & ={1\over{\sqrt{\pi}}}\int_0^{\infty}{|a|^{m/2}A_1(a) \mathcal{W}_{t|a|}(\sqrt{|a|}x,\sqrt{|a|}y){{dt}\over{\sqrt{t}}}},
    \quad  x,y\in\R^m, \quad x \neq y.
\end{align*}
Note that $R_1(x,y;a) =|a|^{m/2}R_1(\sqrt{|a|}x,\sqrt{|a|}y;1)$, $x,y\in\R^m$.\\

Next, we collect some estimates concerning Hermite and Laguerre heat kernels.
Let $\beta \geq 1/2$ and $u,v \in (0,\infty)$. It is convenient to write
$W_t$ and $W_t^\beta$ in terms of the functions $\fa$ and $\fb$ defined in \eqref{fafb}, as follows
\begin{equation*}\label{Wt}
    W_t(u,v)
        = \frac{\sqrt{\fa}}{\sqrt{2 \pi}} e^{-\fb (u^2+v^2)}e^{\fa u v},
\end{equation*}
and
\begin{align}\label{Wtb}
    W_t^\beta(u,v)
        & = \sqrt{\fa} \sqrt{\fa u v} I_\beta \Big( \fa u v \Big) e^{-\fb (u^2+v^2)} \nonumber \\
        & = W_t(u,v) \sqrt{2\pi \fa u v} I_\beta \Big( \fa u v \Big) e^{-\fa u v}.
\end{align}
Since the asymptotics for the modified Bessel function $I_\beta$ depend on whether its argument is small or large
(see \eqref{eq2.8} and \eqref{eq2.7.1}) it will be useful to consider the next to sets
$$A_1(u,v)=\{t \in (0,\infty) :  \fa  u v \leq 1\}
\quad \text{and} \quad
B_1(u,v)=\{t \in (0,\infty) :  \fa  u v \geq 1\}.$$

\begin{Lem}\label{Lem W}
    Let $\beta \geq 1/2$, $0<\eps<1$ and $t,u,v \in (0,\infty)$. Then, \\ \quad
    \begin{itemize}
        \item[$(a)$] $\displaystyle W_t^\beta(u,v) \leq C W_t(u,v) \leq C \frac{e^{-c|u-v|^2/t}}{\sqrt{t}}.$ \\ \quad  \\

        \item[$(b)$] $\displaystyle \Big|W_t^{\beta}(u,v)-W_t(u,v)\Big|
                                        \leq C e^{-(1-\eps)t} \left\{ \begin{array}{ll}
                                                                \displaystyle {{e^{-c|u-v|^2/t}}\over{\sqrt{t}}}, & t \in A_1(u,v), \\
                                                                \\
                                                                \displaystyle {{1}\over{\fa  u v}} {{e^{-c|u-v|^2/t}}\over{\sqrt{t}}}, & t \in B_1(u,v).
                                                            \end{array}\right.$ \\ \quad  \\

        \item[$(c)$] $\displaystyle \Big| \partial_{u}^{\ell} W_t(u,v)\Big|
                            \leq C \frac{e^{-c|u-v|^2/t}}{t^{(\ell+1)/2}}, \quad \ell =1,2.$ \\ \quad  \\

        \item[$(d)$] $\displaystyle \Big| u \partial_{u}W_t(u,v)+v \partial_{v}W_t(u,v) \Big|
                            \leq C {{e^{-c|u-v|^2/t}}\over{\sqrt{t}}}.$ \\ \quad  \\

        \item[$(e)$] $\displaystyle \Big| u \partial_{u}W_t^{\beta}(u,v)+v \partial_{v}W_t^{\beta}(u,v) \Big|
                            \leq C \Big(1 + e^{(1+\eps)t} \chi_{\{u/2<v<2u\}}(v) \Big) {{e^{-c|u-v|^2/t}}\over{\sqrt{t}}}.$ \\ \quad  \\

        \item[$(f)$] $\displaystyle \Big|u\partial_{u}[W_t^{\beta}-W_t](u,v)+v\partial_{v}[W_t^{\beta}-W_t](u,v)\Big|
                            \leq C {{e^{-(1-\eps)t}}\over{(\fa u v)^{1/4}}} \frac{e^{-c|u-v|^2/t}}{\sqrt{t}}.$ \\ \quad  \\
    \end{itemize}
\end{Lem}

\begin{proof}[Proof of  $(a)$]
    The first inequality is a consequence of \eqref{Wtb} together with \eqref{eq2.8} and \eqref{eq2.7.1}. The second one
    follows easily from \eqref{Hermitekernel} and \eqref{ab}.
\end{proof}

\begin{proof}[Proof of  $(b)$]
    It is straightforward from the relation \eqref{Wtb}, Lemma~\ref{Lem W}, $(a)$; and \eqref{eq2.8}.
\end{proof}

\begin{proof}[Proof of  $(c)$]
    It is enough to note that,
    \begin{align*}
        \Big| \partial_{u} W_t(u,v) \Big|
        = & {\sqrt{\fa}\over{2 \sqrt{\pi}}} \exp\left(-{1\over 4}\left[(u-v)^2{{1+e^{-2t}}\over{1-e^{-2t}}}+(u+v)^2{{1-e^{-2t}}\over{1+e^{-2t}}}\right]\right) \\
        & \times  \Big| (u-v){{1+e^{-2t}}\over{1-e^{-2t}}}+(u+v){{1-e^{-2t}}\over{1+e^{-2t}}} \Big| \\
        \leq & C \frac{e^{-c|u-v|^2/t}}{t}.
    \end{align*}
    We can proceed similarly when we take two derivatives.
\end{proof}

\begin{proof}[Proof of  $(d)$]
    Observe that
    \begin{align}\label{28.1}
         \Big|u \partial_{u}W_t(u,v)+v\partial_{v}W_t(u,v)\Big|
       \leq &   C \sqrt{\fa} \exp\left(-{1\over 4}\left[(u-v)^2{{1+e^{-2t}}\over{1-e^{-2t}}}+(u+v)^2{{1-e^{-2t}}\over{1+e^{-2t}}}\right]\right) \nonumber \\
        &  \times  \left((u-v)^2{{1+e^{-2t}}\over{1-e^{-2t}}}+(u+v)^2{{1-e^{-2t}}\over{1+e^{-2t}}}\right) \nonumber \\
      \leq  &   C {{e^{-c|u-v|^2/t}}\over{\sqrt{t}}}.
    \end{align}
\end{proof}

\begin{proof}[Proof of  $(e)$]
From \eqref{eq2.7} we deduce that,
   \begin{align} \label{eq2.8.1.Z2}
        \partial_u W_t^{\beta}(u,v)
            = &  \sqrt{\fa } \partial_u \Big[\Big(\fa uv \Big)^{\beta + 1/2} e^{-\fb (u^2+v^2)}
                    \Big(\fa  uv \Big)^{-\beta} I_{\beta}\Big(\fa uv \Big)\Big] \nonumber \\
            = &  \sqrt{\fa } \Big[(\beta + 1/2) \fa  v \Big( \fa  uv\Big)^{-1/2} I_{\beta}\Big(\fa  uv\Big)
                    - \sqrt{\fa uv} I_{\beta}\Big(\fa uv\Big) 2 u \fb  \nonumber \\
              &  + \fa  v \sqrt{\fa  uv} I_{\beta +1}\Big( \fa  uv \Big)\Big] e^{-\fb  (u^2+v^2)}.
    \end{align}
Hence
   \begin{align*}
        u\partial_{u}W_t^{\beta}( u , v ) + v\partial_{v}W_t^{\beta}( u , v )
       = &\sqrt{\fa }\Big[\sqrt{\fa u v }I_{\beta}\Big( \fa u v \Big)
            \Big((2\beta + 1)-2 \fb( u ^2+ v ^2)\Big) \\
       & +2\Big( \fa u v \Big)^{3/2}I_{\beta +1}\Big( \fa u v \Big)\Big]
            e^{- \fb  ( u ^2+ v ^2)} \\
         = &\sqrt{2\pi} W_t ( u , v )\Big[\sqrt{\fa u v }I_{\beta}\Big( \fa u v \Big)
            \Big((2\beta + 1)-2 \fb ( u ^2+ v ^2)\Big)\\
       &+2\Big( \fa u v \Big)^{3/2}I_{\beta +1}\Big( \fa u v \Big) \Big] e^{-\fa u v }.
   \end{align*}
By \eqref{eq2.8}, \eqref{eq2.7.1} and \eqref{Hermitekernel} we obtain
        \begin{align*} 
             \left| u\partial_{u}W_t^{\beta}( u , v ) + v\partial_{v}W_t^{\beta}( u , v ) \right|
                    \leq & C \Big(1+ \fb (u^2+ v ^2) + \fa u v \Big){{e^{-c| u - v |^2/(1-e^{-2t})}}\over{\sqrt{t}}} \nonumber \\
             \leq & C \frac{e^{-c| u - v |^2/t}}{\sqrt{t}}, \quad v \in (0,u/2] \cup [2u,\infty).
        \end{align*}
    Suppose now that $0< u /2< v <2 u $. Applying \eqref{eq2.7.1} and Lemma~\ref{Lem W}, $(a)$; we get for each $t \in A_1(u,v)$,
    \begin{align*} 
             \left|u\partial_{u}W_t^{\beta}( u , v ) + v\partial_{v}W_t^{\beta}( u , v )\right|
                & \leq C \Big( 1 + \fb u^2\Big) e^{-c \fa u^2} W_t( u , v ) \\
                & \leq C ( 1 + e^{2t})  W_t( u , v )
                \leq C ( 1 + e^{(1+\eps)t}) {{e^{-c| u - v |^2/t}}\over{\sqrt{t}}}.
        \end{align*}
    On the other hand, \eqref{eq2.8} with $n=0$ give us ,for $t \in B_1(u,v)$,
    \begin{align*}
       & \left| -2 \fb  ( u ^2 +  v ^2)\sqrt{\fa u v }I_{\beta}\Big( \fa u v \Big)
            +2\Big( \fa u v \Big)^{3/2}I_{\beta +1}\Big( \fa u v \Big)\right| e^{-\fa u v } \\
       & \qquad \leq C \left(\left| -2 \fb ( u ^2+ v ^2)+ 2 \fa  u  v  \right|+ \frac{2 \fb ( u ^2+ v ^2) + 2 \fa  u  v }{\fa u v }\right) \\
        & \qquad \leq C \left(( u + v )^2{{1-e^{-2t}}\over{1+e^{-2t}}}+( u - v )^2{{1+e^{-2t}}\over{1-e^{-2t}}}+{{( u ^2+ v ^2)e^{2t}+ u  v }\over{ u  v }}\right).
   \end{align*}
 Thus, for every $t \in B_1(u,v)$,
        \begin{align*}
            & \left| u\partial_{u}W_t^{\beta}( u , v ) + v\partial_{v}W_t^{\beta}( u , v ) \right|
                \leq C(1+e^{(1+\eps)t}){{e^{-c| u - v |^2/t}}\over{\sqrt{t}}}.
        \end{align*}
\end{proof}

\begin{proof}[Proof of  $(f)$]
     According to \eqref{eq2.8.1.Z2} we get,
     \begin{align} \label{eq2.8.1.Z18}
        & u\partial_u[W_t^{\beta}-W_t](u,v)+v\partial_v[W_t^{\beta}-W_t](u,v)  \nonumber \\
        & \qquad = u\partial_{u}\left[W_t(u,v)\left(\sqrt{2\pi \fa u v }I_{\beta}\Big(\fa uv\Big) e^{-\fa uv}-1\right)\right] \nonumber\\
        & \qquad \qquad + v\partial_{v}\left[W_t(u,v)\left(\sqrt{2\pi \fa u v }I_{\beta}\Big(\fa uv\Big) e^{-\fa uv}-1\right)\right] \nonumber\\
        & \qquad = \Big[u \partial_uW_t(u,v)+v\partial_vW_t(u,v)\Big]\left(\sqrt{2\pi \fa u v }I_{\beta}\Big(\fa uv\Big) e^{-\fa uv}-1\right) \nonumber \\
        & \qquad \qquad + W_t(u,v)(u\partial_{u}+v\partial_{v})\left(\sqrt{2\pi \fa u v }I_{\beta}\Big(\fa uv\Big)e^{-\fa uv}\right) \nonumber \\
        & \qquad = {{\sqrt{\fa}}\over 2 \sqrt{\pi}}  \left[(u-v)^2{{1+e^{-2t}}\over{1-e^{-2t}}}+(u+v)^2{{1-e^{-2t}}\over{1+e^{-2t}}}\right] \nonumber \\
        & \qquad \qquad \times \exp\left(-{1\over 4}\left[(u-v)^2{{1+e^{-2t}}\over{1-e^{-2t}}}+(u+v)^2{{1-e^{-2t}}\over{1+e^{-2t}}}\right]\right)
                \left(\sqrt{2\pi \fa a u v }I_{\beta}\Big(\fa a u v \Big)e^{-\fa a u v }-1\right) \nonumber \\
        & \qquad \qquad +  W_t(u,v)(u\partial_{u}+v\partial_{v})\left(\sqrt{2\pi \fa u v }I_{\beta}\Big(\fa uv\Big)e^{-\fa uv}\right).
     \end{align}
     By \eqref{eq2.7} we have that
      \begin{align}\label{eq2.8.1.Z19}
         \frac{d}{dz} \Big[ \sqrt{2\pi z} I_\beta(z)e^{-z}\Big]
            &= \sqrt{2\pi}  \frac{d}{dz} \Big[ z^{\beta+1/2} z^{-\beta} I_\beta(z)e^{-z}\Big] \nonumber \\
        & = \frac{\beta+1/2}{z} \sqrt{2 \pi z} I_\beta(z)e^{-z} + \sqrt{2 \pi z} I_{\beta+1}(z)e^{-z} - \sqrt{2 \pi z} I_\beta(z)e^{-z}, \quad  z \in (0,\infty).
    \end{align}
     Then, by using \eqref{eq2.7.1} we obtain,
    \begin{equation}\label{eq2.8.1.Z20}
        \left| \frac{d}{dz} \Big( \sqrt{2\pi z} I_\beta(z)e^{-z}\Big) \right|
            \leq C z^{\beta-1/2}, \quad  z \in (0,1).
    \end{equation}
      From \eqref{eq2.8} and \eqref{eq2.8.1.Z19} we deduce that, for every $n=2,3, \dots,$
    \begin{align*}
        & \frac{d}{dz} \Big( \sqrt{2 \pi z} I_\beta(z) e^{-z}\Big)
            = \sqrt{2 \pi} \Big( \frac{\beta+1/2}{z} z^{1/2} I_\beta(z) e^{-z} + z^{1/2} I_{\beta+1}(z) e^{-z} - z^{1/2} I_\beta(z) e^{-z}\Big) \\
        & \qquad = \frac{\beta+1/2}{z} \Big( \sum_{r=0}^n (-1)^r\frac{ [\beta,r] }{(2z)^{r}} + \mathcal{O}\Big(\frac{1}{z^{n+1}}\Big)\Big)
                    + \sum_{r=0}^n (-1)^r \frac{[\beta +1 ,r]}{ (2z)^{r}} + \mathcal{O}\Big(\frac{1}{z^{n+1}}\Big) \\
        & \qquad \qquad - \sum_{r=0}^n (-1)^r \frac{[\beta ,r]}{ (2z)^{r}} + \mathcal{O}\Big(\frac{1}{z^{n+1}}\Big)
    \end{align*}
    \begin{align*}
        & \qquad = \sum_{r=1}^n (-1)^r \frac{-(2\beta+1)[\beta,r-1] + [\beta+1,r] - [\beta,r]}{ (2z)^{r}} + \mathcal{O}\Big(\frac{1}{z^{n+1}}\Big), \quad  z \in (0,\infty).
    \end{align*}
    Since,
    $$-(2\beta+1)[\beta,r-1] + [\beta+1,r] - [\beta,r]
        = 2(r-1)[\beta,r-1], \quad  r=1,2,3, \dots, $$
    we get
    \begin{equation*}
        \frac{d}{dz} \Big( \sqrt{2 \pi z} I_\beta(z) e^{-z}\Big)
            = \sum_{r=2}^n (-1)^r\frac{ 2(r-1)[\beta,r-1]}{(2z)^{r}} + \mathcal{O}\Big(\frac{1}{z^{n+1}}\Big), \quad  z \in (0,\infty),
    \end{equation*}
    for every $n=2,3, \dots$. Then,
    \begin{equation}\label{eq2.8.1.Z21}
        \frac{d}{dz} \Big( \sqrt{2 \pi z} I_\beta(z) e^{-z}\Big)
            = \mathcal{O}\Big(\frac{1}{z^{2}}\Big), \quad  z \in (0,\infty).
    \end{equation}
    According to \eqref{eq2.8} and \eqref{eq2.7.1}  it follows that
    \begin{align*}
        &\Big| \sqrt{\fa}\Big[( u - v )^2{{1+e^{-2t}}\over{1-e^{-2t}}}+( u + v )^2{{1-e^{-2t}}\over{1+e^{-2t}}}\Big]
        \exp\Big(-{1\over 4}\Big[( u - v )^2{{1+e^{-2t}}\over{1-e^{-2t}}}+( u + v )^2{{1-e^{-2t}}\over{1+e^{-2t}}}\Big]\Big) \nonumber \\
        & \qquad  \qquad \times \Big(\sqrt{2\pi \fa u v }I_{\beta}\Big(\fa u v \Big)e^{-\fa u v }-1\Big)\Big| \nonumber \\
        & \qquad \leq C {{e^{-(1-\eps)t}}\over{(\fa u v )^{1/4}}}  \frac{e^{-c| u - v |^2/t}}{\sqrt{t}}, \quad  u , v ,t\in(0,\infty).
     \end{align*}
     Moreover, \eqref{eq2.8.1.Z20} and \eqref{eq2.8.1.Z21} imply that
     \begin{align*}
        & \Big|W_t( u , v )( u \partial_{u}+ v \partial_{v})\Big(\sqrt{2\pi \fa u v }I_{\beta}(\fa u v )e^{-\fa u v }\Big)\Big|
        \leq C {{e^{-(1-\eps)t}}\over{(\fa u v )^{1/4}}} \frac{e^{-c|u - v|^2/t}}{\sqrt{t}}.
    \end{align*}
\end{proof}

Recall that, for every $t,u,v \in(0,\infty)$,
$$G_t^{\beta}(u,v)=A_{\beta}(1)W_t^{\beta}(u,v)
\quad \text{and} \quad
G_t(u,v)=A(1)W_t(u,v),$$
being $A_\beta(1)$ and $A(1)$ the usual derivatives in the Laguerre and Hermite setting, respectively (see \eqref{3.1} and \eqref{13.1}).

\begin{Lem}\label{Lem G}
    Let $\beta \geq 1/2$, $0<\eps<1$ and $t,u,v \in (0,\infty)$. Then, \\ \quad
    \begin{itemize}
        \item[$(a)$] $\displaystyle \Big| G_t(u,v) \Big|
                            \leq C e^{-(3-\eps)t} \frac{e^{-c|u-v|^2/t}}{t}.$ \\ \quad  \\

        \item[$(b)$] $\displaystyle \Big| G_t^{\beta}(u,v)-G_t(u,v) \Big|
                                        \leq C  \left\{ \begin{array}{ll}
                                                                \displaystyle \frac{e^{-(3/2-\eps)t}}{\sqrt{u}} {{e^{-c|u-v|^2/t}}\over{t^{3/4}}}, & u/2 < v < 2u  , \\
                                                                \\
                                                                \displaystyle e^{-(3-\eps)t} \max\{u,v\} {{e^{-c|u-v|^2/t}}\over{t^{3/2}}}, & v \in (0,u/2] \cup [2u,\infty).
                                                            \end{array}\right.$ \\ \quad  \\

        \item[$(c)$] $\displaystyle \Big|u\partial_{u}G_t(u,v) + v\partial_{v}G_t(u,v)\Big|
                            \leq C \frac{e^{-c|u-v|^2/t}}{t}.$ \\ \quad  \\

        \item[$(d)$] $\displaystyle \Big|\partial_{u}[G_t^{\beta}-G_t](u,v)\Big|
                            \leq C \frac{e^{-c(u^2+v^2)/t}}{t^{3/2}}, \quad t \in A_1(u,v).$ \\ \quad  \\

        \item[$(e)$] $\displaystyle \Big|u \partial_{u}[G_t^{\beta}-G_t](u,v) + v\partial_{v}[G_t^{\beta}-G_t](u,v) \Big|
                    \leq C \Big( \frac{u^{1/4}}{v^{3/4}} + \frac{v^{1/4}}{u^{3/4}} \Big) \frac{e^{-c|u-v|^2/t}}{t^{3/4}}, \quad t \in B_1(u,v).$ \\ \quad  \\

        \item[$(f)$] $\displaystyle \Big|u \partial_{u}[G_t^{\beta}-G_t](u,v) + v\partial_{v}[G_t^{\beta}-G_t](u,v) \Big|
        \leq C \frac{e^{-c\max\{u,v\}^2/t}}{t}, \quad v \in (0,u/2] \cup [2u,\infty).$ \\ \quad  \\

        \item[$(g)$] $\displaystyle \Big| u\partial_{u}^2[G_t^{\beta}-G_t](u,v) \Big|
                                        \leq C  \left\{ \begin{array}{ll}
                                                                \displaystyle \frac{e^{-c(u^2+v^2)/t}}{t^{3/2}}, & t \in A_1(u,v)  , \\
                                                                \\
                                                                \displaystyle \frac{e^{-c|u-v|^2/t}}{t^{3/2}}, & t \in B_1(u,v).
                                                            \end{array}\right.$ \\ \quad  \\
    \end{itemize}
\end{Lem}

\begin{proof}[Proof of  $(a)$]
     We have that
    \begin{equation}\label{Gt}
        G_t(u,v)= \Big(\fa v + (1-2\fb )u \Big) W_t(u,v).
    \end{equation}
    Thus,
         \begin{align*}
         \left|G_t(u,v)\right| & \leq C e^{-3t} {{|u-v|+u(1-e^{-2t})}\over{1-e^{-4t}}}{{e^{-c(|u-v|^2/(1-e^{-2t})+|u+v|^2(1-e^{-2t}))}}\over{\sqrt{t}}}
        \leq C e^{-(3-\eps)t} {{e^{-c|u-v|^2/t}}\over t}.
        \end{align*}
\end{proof}

\begin{proof}[Proof of  $(b)$]
        By taking into account that
        $$A_{\beta}(1)
        = \frac{d}{du}  + u - \frac{\beta+1/2}{u}
        = u^{\beta+1/2} \frac{d}{du} u^{-\beta-1/2}+u,$$
     \eqref{eq2.7} leads to
    \begin{align}\label{Gtb}
        G_t^{\beta}(u,v)
            = &  u W_t^{\beta}(u,v) + \fa ^{\beta+1}(uv)^{\beta+1/2} \partial_{u} \Big[ e^{-\fb (u^2+v^2)}
                    \Big(\fa uv \Big)^{-\beta} I_{\beta}\Big(\fa uv \Big)\Big] \nonumber \\
            = & \fa ^{\beta+1} (uv)^{\beta+1/2} e^{-\fb (u^2+v^2)} \nonumber \\
            & \times \Big[ \Big(\fa uv \Big)^{-\beta} I_{\beta}\Big(\fa uv \Big)(1-2\fb )u  +  \Big(\fa uv \Big)^{-\beta} I_{\beta+1}\Big(\fa uv \Big)v\fa   \Big] \nonumber \\
            = & u W_t(u,v) \sqrt{2\pi \fa uv} I_{\beta}\Big(\fa uv \Big) e^{-\fa uv} \nonumber \\
            & + W_t(u,v) \Big[ -2\fb u\sqrt{2\pi \fa uv}I_{\beta}\Big(\fa uv \Big) e^{-\fa uv} \nonumber \\
            & + \fa v\sqrt{2\pi \fa uv}I_{\beta+1}\Big(\fa uv \Big)e^{-\fa uv} \Big], \quad  t,u,v \in (0,\infty).
    \end{align}
    Putting together \eqref{Gt} and \eqref{Gtb} we arrive at
    \begin{align}\label{eq2.8.1.Z1}
        G_t^{\beta}(u,v)-G_t(u,v)
            = & u W_t(u,v) \Big( \sqrt{2\pi \fa uv}I_{\beta}\Big(\fa uv \Big)e^{-\fa uv} -1\Big) \nonumber \\
            & + W_t(u,v) \Big[ -2\fb u \Big( \sqrt{2\pi \fa uv}I_{\beta}\Big(\fa uv \Big)e^{-\fa uv} -1\Big)\nonumber \\
            & + \fa  v \Big(\sqrt{2\pi \fa uv}I_{\beta+1}\Big(\fa uv \Big)e^{-\fa uv}-1\Big)\Big].
    \end{align}
    First of all, assume that $u/2<v<2u$. From \eqref{eq2.8} with $n=0$ it follows that,
        \begin{align}\label{eq2.8.1.Z7}
            \left|G_t^{\beta}(u,v)-G_t(u,v)\right|
                & \leq C  {|1-2\fb |u+\fa v  \over{(\fa u v)^{3/4}}}W_t(u,v)
                \leq C  \frac{\fa^{1/4}}{\sqrt{u}} W_t(u,v) \nonumber\\
                & \leq C  \frac{e^{-(3/2-\eps)t}}{\sqrt{u}} {{e^{-c|u-v|^2/t}}\over{t^{3/4}}}, \quad t \in B_1(u,v).
        \end{align}
        Also \eqref{eq2.7.1} leads to,
        \begin{align}\label{eq2.8.1.Z8}
             \left|G_t^{\beta}(u,v)-G_t(u,v)\right|
                &\leq C \Big(|1-2\fb |u+\fa v \Big) W_t(u,v)
                    \leq C  {\fa u  \over{(\fa u v)^{3/4}}}W_t(u,v)\nonumber \\
                & \leq C  \frac{e^{-(3/2-\eps)t}}{\sqrt{u}} {{e^{-c|u-v|^2/t}}\over{t^{3/4}}}, \quad t \in A_1(u,v).
        \end{align}
        Suppose now that $v \in (0,u/2] \cup [2u,\infty)$. By proceeding as in \eqref{eq2.8.1.Z7} and \eqref{eq2.8.1.Z8} we can obtain
         $$\left|G_t^{\beta}(u,v)-G_t(u,v)\right|
            \leq C  \fa (u+v) W_t(u,v)
            \leq C  e^{-(3-\eps)t} \max\{u,v\} {{e^{-c|u-v|^2/t}}\over{t^{3/2}}}.$$
\end{proof}

\begin{proof}[Proof of  $(c)$]
    By \eqref{Gt} we can write
    \begin{align*}
        u\partial_{u}G_t(u,v) + v\partial_{v}G_t(u,v)
            = G_t(u,v) + \Big( \fa v + (1-2\fb)u \Big) \Big( u\partial_{u}W_t(u,v) + v\partial_{v}W_t(u,v) \Big).
    \end{align*}
    Since,
    $$\fa v + (1-2\fb)u
            = \fa \Big( (v-u) + (1-e^{-2t})u\Big),$$
    Lemma~\ref{Lem G}, $(a)$, and the first estimate in \eqref{28.1} allow us to deduce the desired conclusion.
\end{proof}

\begin{proof}[Proof of  $(d)$]
        Identity \eqref{eq2.8.1.Z1} allows us to write
    \begin{align*}
         \partial_{u}[G_t^{\beta}-G_t](u,v)
        =&   (1-2\fb )\left[ W_t(u,v) \Big( \sqrt{2\pi \fa u v}I_{\beta}\Big(\fa u v \Big)e^{-\fa u v} -1 \Big) \right. \\
        & \qquad + u\partial_{u}W_t(u,v) \Big( \sqrt{2\pi \fa u v}I_{\beta}\Big(\fa u v \Big)e^{-\fa u v} -1\Big) \\
        &\left. \qquad -uW_t(u,v)\partial_{u}\Big( \sqrt{2\pi \fa u v}I_{\beta}\Big(\fa u v \Big)e^{-\fa u v}\Big)\right] \\
        &  +\fa v\left[ \partial_{u}W_t(u,v) \Big( \sqrt{2\pi \fa u v}I_{\beta+1}\Big(\fa u v \Big)e^{-\fa u v} -1\Big) \right. \\
        &\left. \qquad -W_t(u,v)\partial_{u}\Big( \sqrt{2\pi \fa u v}I_{\beta+1}\Big(\fa u v \Big)e^{-\fa u v}\Big)\right].
    \end{align*}
    By taking into account \eqref{eq2.7.1} and \eqref{eq2.8.1.Z20}  we conclude
    \begin{align*}
        \left| \partial_{u}[G_t^{\beta}-G_t](u,v) \right|
       & \leq C \fa \Big((1+v+u)W_t(u,v)+(u+v)|\partial_{u}W_t(u,v)|\Big) \\
       &\leq  C {e^{-c(u^2+v^2)/t}\over{t^{3/2}}}, \quad t\in A_1(u,v).
    \end{align*}
\end{proof}

\begin{proof}[Proof of  $(e)$]
According to \eqref{eq2.7} and the identity \eqref{eq2.8.1.Z1} we can write
    \begin{align*}
        & u \partial_{u}  \Big( G_t^{\beta}(u,v) - G_t(u,v) \Big) + v \partial_{v}  \Big( G_t(u,v) - G_t^{\beta}(u,v) \Big) \nonumber \\
        & \qquad =(1-2\fb ) \Big[ u W_t(u,v) \Big( \sqrt{2\pi \fa uv}I_{\beta}\Big(\fa uv \Big)e^{-\fa uv} -1\Big) \nonumber \\
        & \qquad \qquad \qquad+ u^2 \partial_{u}  W_t(u,v) \Big( \sqrt{2\pi \fa uv}I_{\beta}\Big(\fa uv \Big)e^{-\fa uv} -1\Big) \nonumber \\
        & \qquad \qquad \qquad- u^2  W_t(u,v) \partial_{u}  \Big( \sqrt{2\pi \fa uv}I_{\beta}\Big(\fa uv \Big)e^{-\fa uv} \Big) \nonumber \\
        & \qquad \qquad \qquad+ uv \partial_{v}  W_t(u,v) \Big( \sqrt{2\pi \fa uv}I_{\beta}G_t^{\beta}-\Big(\fa uv \Big)e^{-\fa uv} -1\Big) \nonumber \\
        & \qquad \qquad \qquad- uv  W_t(u,v) \partial_{v} \Big( \sqrt{2\pi \fa uv}I_{\beta}\Big(\fa uv \Big)e^{-\fa uv} \Big) \Big]\nonumber \\
        & \qquad \qquad + \fa  \Big[ uv \partial_{u}  W_t(u,v) \Big( \sqrt{2\pi \fa uv}I_{\beta+1}\Big(\fa uv \Big)e^{-\fa uv} -1\Big) \nonumber \\
        & \qquad \qquad \qquad- uv  W_t(u,v) \partial_{u} \Big( \sqrt{2\pi \fa uv}I_{\beta+1}\Big(\fa uv \Big)e^{-\fa uv} \Big) \nonumber \\
        & \qquad \qquad \qquad+ v W_t(u,v) \Big( 1 - \sqrt{2\pi \fa uv}I_{\beta+1}\Big(\fa uv \Big)e^{-\fa uv} \Big) \nonumber \\
        & \qquad \qquad \qquad+ v^2 \partial_{v}  W_t(u,v) \Big( \sqrt{2\pi \fa uv}I_{\beta+1}\Big(\fa uv \Big)e^{-\fa uv} -1\Big)\nonumber \\
        & \qquad \qquad \qquad- v^2  W_t(u,v) \partial_{v} \Big( \sqrt{2\pi \fa uv}I_{\beta+1}\Big(\fa uv \Big)e^{-\fa uv} \Big)\Big]  \nonumber \\
        & \qquad = \sum_{j=1}^{10} H_j(t,u,v).
    \end{align*}
    Let $t \in B_1(u,v)$. By using \eqref{eq2.8} we obtain,
    \begin{align*}
        \Big| H_1(t,u, v) + H_8(t,u,v)\Big|
            \leq C \frac{|1-2\fb|}{(\fa u v)^{3/4}} (u+v)  W_t(u,v)
            \leq C \Big( \frac{u^{1/4}}{v^{3/4}} + \frac{v^{1/4}}{u^{3/4}} \Big) \frac{e^{-c|u-v|^2/t}}{t^{3/4}}.
    \end{align*}
    Asymptotic behavior \eqref{eq2.8} leads also to
    \begin{align*}
         \Big|H_2(t,u,v) + H_4(t,u,v)\Big|
            &\leq C \frac{\fa u}{\Big(\fa uv \Big)^{3/4}} \Big| u \partial_{u}  W_t(u,v) + v \partial_{v}  W_t(u,v)  \Big| \\
            &  = C \frac{\fa ^{3/4}u^{1/4}}{v^{3/4}} [\fb (u^2+v^2)-\fa uv]e^{-\fb (u^2+v^2)+\fa uv} \\
            &  \leq C  \frac{u^{1/4}}{v^{3/4}} \frac{e^{-c|u-v|^2/t}}{t^{3/4}}.
    \end{align*}
    Also, we have that
    \begin{align*}
         \Big|H_6(t,u,v) + H_9(t,u,v)\Big|
            &\leq C \frac{\fa v}{\Big(\fa uv \Big)^{3/4}} \Big| u \partial_{u}  W_t(u,v) + v \partial_{v}  W_t(u,v)  \Big|
             \leq C  \frac{v^{1/4}}{u^{3/4}} \frac{e^{-c|u-v|^2/t}}{t^{3/4}}.
    \end{align*}
    According to \eqref{eq2.8.1.Z21} we get,
       \begin{align*}
        & \Big| H_3(t,u,v) + H_5(t,u,v) \Big|
            \leq C \frac{\fa ^2 u^2v}{\Big(\fa uv \Big)^{7/4}} W_t(u,v)
         \leq C \frac{u^{1/4}}{v^{3/4}} \frac{e^{-c|u-v|^2/t}}{t^{3/4}}.
    \end{align*}
    Similarly we can show that
    \begin{align*}
        &  \Big| H_7(t,\sqrt{a}u, \sqrt{a}v) + H_{10}(\sqrt{a}u, \sqrt{a}v,t)\Big|
         \leq C \frac{v^{1/4}}{u^{3/4}} \frac{e^{-c|u-v|^2/t}}{t^{3/4}}.
    \end{align*}
\end{proof}

\begin{proof}[Proof of  $(f)$]
    According to \eqref{eq2.8.1.Z18} and by using \eqref{eq2.8}, \eqref{eq2.7.1}, \eqref{eq2.8.1.Z20} and  \eqref{eq2.8.1.Z21} we obtain
    \begin{align*}
        & \left| u \partial_{u}  \Big( G_t(u,v) - G_t^{\beta}(u,v) \Big) + v \partial_{v}  \Big( G_t(u,v) - G_t^{\beta}(u,v) \Big) \right| \nonumber \\
        & \qquad \leq C \fa  \Big\{ (u+v) W_t(u,v) + u^2 |\partial_{u} W_t(u,v)| + v^2 |\partial_{v} W_t(u,v)|
                    + uv (|\partial_{u} W_t(u,v)| + |\partial_{v} W_t(u,v)|)\Big\}  \nonumber \\
        & \qquad \leq  \frac{C}{t^{3/2}}\left[u+v+{{u^2+v^2}\over{\sqrt{t}}}\right] e^{-c|u-v|^2/t}
         \leq  \frac{C}{t}\left\{
            \begin{array}{l}
            e^{-cu^2/t},\;\;0<v<u/2, \\
            \\
            e^{-cv^2/t},\;\;0<2u<v.
          \end{array}\right.
    \end{align*}
\end{proof}

\begin{proof}[Proof of  $(g)$]
 We have that
     \begin{align}\label{eqG28}
         \partial_{u}^2 [ G_t^{\beta}- G_t](u,v)
        =&   (1-2\fb ) \Big\{ 2 \partial_{u} \Big[ W_t(u,v)  \Big( \sqrt{2\pi \fa u v} I_{\beta}\Big(\fa u v \Big) e^{-\fa u v} -1 \Big) \Big] \nonumber \\
         & \qquad + u \Big[ \partial^2_{u} W_t(u,v) \Big( \sqrt{2\pi \fa u v} I_{\beta}\Big(\fa u v \Big) e^{-\fa u v} -1\Big) \nonumber \\
        &  \qquad - 2 \partial_{u} W_t(u,v) \partial_{u} \Big( \sqrt{2\pi \fa u v} I_{\beta}\Big(\fa u v \Big) e^{-\fa u v} \Big) \nonumber \\
        &  \qquad - W_t(u,v) \partial^2_{u} \Big( \sqrt{2\pi \fa u v} I_{\beta}\Big(\fa u v \Big) e^{-\fa u v} \Big)\Big] \Big\} \nonumber \\
        &  + \fa  v \Big\{ \partial^2_{u} W_t(u,v) \Big( \sqrt{2\pi \fa u v} I_{\beta+1}\Big(\fa u v \Big) e^{-\fa u v} -1\Big) \nonumber \\
        &  \qquad - 2 \partial_{u} W_t(u,v) \partial_{u} \Big( \sqrt{2\pi \fa u v} I_{\beta+1}\Big(\fa u v \Big) e^{-\fa u v} \Big) \nonumber \\
        &  \qquad - W_t(u,v) \partial^2_{u} \Big( \sqrt{2\pi \fa u v} I_{\beta+1}\Big(\fa u v \Big) e^{-\fa u v} \Big)\Big\}.
    \end{align}
    By using \eqref{eq2.7} and \eqref{22b} we get
    \begin{equation}\label{eqG29}
        \frac{d^2}{dz^2} \Big[ \sqrt{z} I_{\beta}(z)e^{-z} \Big]
            = e^{-z} \Big[ \Big( \frac{4\beta^2-1}{4z^2} + 2 - \frac{2\beta + 1}{z}\Big)  \sqrt{z} I_{\beta}(z) -  2\sqrt{z} I_{\beta+1}(z)\Big], \quad  z \in (0,\infty).
    \end{equation}
    According to \eqref{eq2.8} it follows that
    \begin{equation}\label{eqG30}
        \frac{d^2}{dz^2} \Big[ \sqrt{z} I_{\beta}(z)e^{-z} \Big]
            = \mathcal{O}\Big( \frac{1}{z^3}\Big), \quad  z \in (0,\infty).
    \end{equation}
    Equalities \eqref{eq2.8}, \eqref{eq2.8.1.Z21}, \eqref{eqG28}, \eqref{eqG30} and Lemma~\ref{Lem W}, $(c)$, lead to
    \begin{align*}
     \Big| u  \partial_{u}^2 [ G_t^{\beta}- G_t](u,v) \Big|
        \leq  &   C u \fa  \Big\{ |\partial_{u} W_t(u,v)| \frac{1}{\Big(\fa u v \Big)^{1/2}} +  W_t(u,v) \frac{\fa v}{\fa u v}  \\
    &  + u |\partial^2_{u} W_t(u,v)|  \frac{1}{\fa u v} + u |\partial_{u} W_t(u,v)| \frac{\fa v}{\Big(\fa u v \Big)^{3/2}}
            + u W_t(u,v) \frac{(\fa v)^2}{\Big(\fa u v \Big)^2}\\
    &   + v  |\partial^2_{u} W_t(u,v)| \frac{1}{\fa u v} + v |\partial_{u} W_t(u,v)| \frac{\fa v}{\Big(\fa u v \Big)^{3/2}}
            + v W_t(u,v) \frac{(\fa v)^2}{\Big(\fa u v \Big)^2} \Big\}\\
   \leq &   C \frac{e^{-c|u-v|^2/t}}{t^{3/2}}, \quad  t \in B_1(u,v), \ 0<u/2<v<2u.
    \end{align*}
    Analogously, when $0<v<u/2$ or $0<2u<v$, we get
     \begin{align*}
         \Big| u  \partial_{u}^2 [ G_t^{\beta}- G_t](u,v) \Big|
       \leq &       C u \fa  \Big\{ |\partial_{u} W_t(u,v)| +  W_t(u,v) \fa v + u |\partial^2_{u} W_t(u,v)|  \\
        &  + u |\partial_{u} W_t(u,v)| \fa v + u W_t(u,v) (\fa v)^2 + v  |\partial^2_{u} W_t(u,v)| \\
        &  + v |\partial_{u} W_t(u,v)| \fa v + v W_t(u,v) (\fa v)^2 \Big\}\\
       \leq &    C \frac{u}{t} \Big\{ \frac{1}{t} + \frac{u+v}{t^{3/2}} + \frac{uv+v^2}{t^{2}} + \frac{uv^2+v^3}{t^{5/2}} \Big\} e^{-c|u-v|^2/t} \\
       \leq &    C \frac{e^{-c|u-v|^2/t}}{t^{3/2}}, \quad  t \in B_1(u,v).
    \end{align*}
    On the other hand, from \eqref{eq2.7.1}, \eqref{eq2.8.1.Z20} and \eqref{eqG29} we deduce, for every $z\in(0,1]$,
    $$\Big|\frac{d^\ell}{dz^\ell} \Big( \sqrt{z} I_{\beta}(z)e^{-z}\Big) \Big| \leq C, \quad \ell =0,1,
        \qquad \Big| \frac{d^2}{dz^2} \Big( \sqrt{z} I_{\beta}(z)e^{-z}\Big) \Big| \leq \frac{C}{z}.$$
    Hence,
    \begin{align*}
         \Big| u \partial_{u}^2 [ G_t^{\beta}- G_t](u,v) \Big|
      \leq  &   C \fa  \Big\{ (1+\fa v^2)W_t(u,v) + (u+v)|\partial_{u} W_t(u,v)| + (u^2+v^2)|\partial_{u}^2 W_t(u,v)|   \Big\}\\
       \leq &    C \frac{e^{-c(u^2+v^2)/t}}{t^{3/2}}, \quad t \in A_1(u,v).
    \end{align*}
\end{proof}

\subsection{The proof of the Theorem}

In this section we prove that the Riesz transform $R_{\alpha, 1}$ is bounded from $L^p((0,\infty)^m\times\R)$ into itself, for every $1<p<\infty$.
The boundedness of the other Riesz Transforms can be showed in a similar way.

The operator $R_{\alpha, 1}$ is defined by
$$R_{\alpha, 1}(f)(x,y)
    =\mathcal{F}_2^{-1}(R_{\alpha, 1}(|u|)(\mathcal{F}_2(f))(x,u))(y), \quad f\in L^2((0,\infty)^m\times\R).$$
Let $1<p<\infty$. We identify $L^p((0,\infty)^m\times\R)$ with $L^p(\R,L^p((0,\infty)^m))$. The Riesz transform $R_{\alpha, 1}$ can be understood
as a Banach valued Fourier multiplier. Indeed, let $f\in C_c^\infty((0,\infty)^m\times\R)\subseteq S(\R,L^p((0,\infty)^m))$,
where $S(\R,L^p((0,\infty)^m))$ denotes the $L^p((0,\infty)^m)$-valued Schwartz functions space.
Then, $\mathcal{F}(f)\in  S(\R,L^p((0,\infty)^m))$;  being $\mathcal{F}$ the ($L^p((0,\infty)^m)$-valued) Fourier transform.
By using Plancherel equality for Laguerre functions expansion we obtain
\begin{equation}\label{A1}
    \|R_{\alpha, 1}(a)\|_{L^2((0,\infty)^m)\rightarrow L^2((0,\infty)^m)}\leq 1, \quad a>0.
\end{equation}
Let $a>0$. From \eqref{eq2.2} we get, for every $g\in C_c^\infty((0,\infty)^m)$,
$$R_{\alpha, 1}(a)(g)(x)
    =\int_0^{\infty}{ R_{\alpha, 1}(x,y;a)g(y)dy}, \quad \mbox{a.e.}\;x\notin supp\;g.$$
Moreover, we have that for $ x,y \in (0,\infty)^m$ and $t>0$,
\begin{align*}
    & A_{\alpha_1}(a)\mathcal{W}_t^{\alpha}(x,y;a)= A_{\alpha_1}(a)W_{ta}^{\alpha_1}(\sqrt{a}x_1,\sqrt{a}y_1)a^{m/2}\prod_{j=2}^m{W_{ta}^{\alpha_j}(\sqrt{a}x_j,\sqrt{a}y_j)} \\
    & \qquad \qquad = a^{(m+1)/2}\left({d\over{d(\sqrt{a}x_1)}}+\sqrt{a}x_1-{{\alpha_1 +{1/2}}\over{\sqrt{a}x_1}}\right)\left[W_{ta}^{\alpha_1}(\sqrt{a}x_1,\sqrt{a}y_1)\right] \prod_{j=2}^m{W_{ta}^{\alpha_j}(\sqrt{a}x_j,\sqrt{a}y_j)} \\
    & \qquad \qquad = a^{(m+1)/2}\big[A_{\alpha_1}(1)W_{ta}^{\alpha_1}(\tilde{x}_1,\tilde{y}_1)\big]_{|\tilde{x}_1=\sqrt{a}x_1, \; \tilde{y}_1=\sqrt{a}y_1}\prod_{j=2}^m{W_{ta}^{\alpha_j}(\sqrt{a}x_j,\sqrt{a}y_j)}.
\end{align*}
Then,
\begin{align*}
    R_{\alpha, 1}(x,y;a)
    & = {{a^{(m+1)/2}}\over{\sqrt{\pi}}}\int_0^{\infty}{\big[A_{\alpha_1}(1)W_{ta}^{\alpha_1}(\tilde{x}_1,\tilde{y}_1)\big]\prod_{j=2}^m{W_{ta}^{\alpha_j}(\tilde{x}_j,\tilde{y}_j)}_{|\tilde{x}=\sqrt{a}x, \; \tilde{y}=\sqrt{a}y}{{dt}\over{\sqrt{t}}}} \\
    & = {{a^{m/2}}\over{\sqrt{\pi}}}\int_0^{\infty}{\big[A_{\alpha_1}(1)W_{s}^{\alpha_1}(\tilde{x}_1,\tilde{y}_1)\big]\prod_{j=2}^m{W_{s}^{\alpha_j}(\tilde{x}_j,\tilde{y}_j)}_{|\tilde{x}=\sqrt{a}x, \; \tilde{y}=\sqrt{a}y}{{ds}\over{\sqrt{s}}}} \\
    & = a^{m/2} R_{\alpha, 1}(\sqrt{a}x,\sqrt{a}y;1), \quad x,y \in (0,\infty)^m, \quad x \neq y.
\end{align*}

Since $R_{\alpha, 1}(x,y;1)$ is a standard Calderón-Zygmund kernel \cite[Proposition 3.1]{NoSt1} we can obtain
\begin{equation}\label{A2}
    |R_{\alpha, 1}(x,y;a)|
        \leq {C\over{|x-y|^m}}, \quad x,y\in (0,\infty)^m,\;x\neq y,
\end{equation}
and
\begin{equation}\label{A3}
    |{\nabla}_xR_{\alpha, 1}(x,y;a)|+|{\nabla}_yR_{\alpha}(x,y;a)|
        \leq {C\over{|x-y|^{m+1}}}, \quad x,y\in (0,\infty)^m,\;x\neq y,
\end{equation}
where the constant $C>0$ does not depend on $a$.

By using Calderón-Zygmund theory we conclude that $R_{\alpha, 1}$ can be extended from
$L^2((0,\infty)^m)\cap L^p((0,\infty)^m)$ to $L^p((0,\infty)^m)$ as a bounded operator from $L^p((0,\infty)^m)$ into itself (as it was said earlier) and
$$\|R_{\alpha, 1}(a)\|_{L^p((0,\infty)^m)\rightarrow L^p((0,\infty)^m)}\leq C,$$
where $C>0$ does not depend on $a$.

We deduce that $R_{\alpha, 1}(|u|)(\mathcal{F}(f)(u))\in L^1(\R,L^p((0,\infty)^m))$ and then
$\mathcal{F}^{-1}(R_{\alpha, 1}(|u|)(\mathcal{F}(f)(u)))\in L^{\infty}(\R,L^p((0,\infty)^m))$. Furthermore, we have that
\begin{equation}\label{eq2.4}
    R_{\alpha, 1}(f)=\mathcal{F}^{-1}(R_{\alpha, 1}(|u|)(\mathcal{F}(f)(u))).
\end{equation}
Indeed, suppose that $g\in  L^1(\R,L^p((0,\infty)^m))$. We understand $g$ as a function defined in $\R \times (0,\infty)^m$.
Let $h\in L^{p'}((0,\infty)^m)$, where $p'=p/(p-1)$. By using some properties of the Bochner integral and Hölder inequality we get
\begin{align*}
    \int_{(0,\infty)^m}{h(z)\left(\int_{\R}{g(y,\cdot)e^{-ixy}dy}\right)(z)dz}
    & =\int_{\R}{\int_{(0,\infty)^m}{h(z)g(y,z)dz}e^{-ixy}dy} \\
    & =\int_{(0,\infty)^m}{h(z)\int_{\R}{g(y,z)e^{-ixy}dy}dz}, \quad x\in\R.
\end{align*}
Then, for every $x\in\R$,
$$\left(\int_{\R}{g(y,\cdot)e^{-ixy}dy}\right)(z)
    =\int_{\R}{g(y,z)e^{-ixy}dy}, \quad \mbox{a.e.}\;z\in (0,\infty)^m.$$
According to \eqref{eq2.4}, since $L^p((0,\infty)^m)$ is a UMD Banach space, we use \cite[Theorem 3.4]{We} to show that $R_{\alpha, 1}$
defines a bounded operator from $L^p(\R,L^p((0,\infty)^m))$ into itself.
It is sufficient to see that the families of operators
$$ \Big\{R_{\alpha, 1}(u)\Big\}_{u>0} \quad  \text{and} \quad  \Big\{u \frac{d}{du}R_{\alpha, 1}(u)\Big\}_{u>0}$$
are $R$-bounded in $L^p((0,\infty)^m)$. It is well-known that if $\{T(u)\}_{u>0}$ is a set
of bounded operators in $L^p((0,\infty)^m)$, then $\mathcal{T}=\{T(u)\}_{u>0}$ is $R$-bounded in $L^p((0,\infty)^m)$
if, and only if, there exists $C>0$ such that
\begin{equation}\label{eq2.22}
    \Big\| \Big( \sum_{j=1}^N |T(u_j)g_j|^2 \Big)^{1/2} \Big\|_{L^p((0,\infty)^m)}
        \leq C \Big\| \Big( \sum_{j=1}^N |g_j|^2 \Big)^{1/2} \Big\|_{L^p((0,\infty)^m)},
\end{equation}
for every sequences $(u_j)_{j=1}^N \subset (0,\infty)$, $(g_j)_{j=1}^N \subset L^p((0,\infty)^m)$ and $N\in\N$.

Suppose that $\{T(u)\}_{u>0}$ is a family of bounded operators in $L^2((0,\infty)^m)$ such that
$$\sup_{u>0} \|T(u)\|_{\mathcal{L}(L^2((0,\infty)^m))}<\infty,$$
being $\mathcal{L}(L^2((0,\infty)^m))$ the space of bounded linear mappings from $L^2((0,\infty)^m)$ into itself.
Moreover, assume that, for every $u>0$ and $g \in C_c^\infty((0,\infty)^m)$,
$$T(u)g(x)
    = \int_{(0,\infty)^m} K_u(x,y) g(y)dy, \quad  x \notin \supp(g),$$
where
$$\sup_{u>0} |K_u(x,y)|
    \leq \frac{C}{|x-y|^m}, \quad  x,y \in (0,\infty)^m, x \neq y,$$
and
$$\sup_{u>0} \Big(|\nabla_x K_u(x,y)| + |\nabla_y K_u(x,y)| \Big)
    \leq \frac{C}{|x-y|^{m+1}}, \quad  x,y \in (0,\infty)^m, x \neq y.$$
Then, by \cite[Theorem 1.3 Chapter XII]{Tor}, \eqref{eq2.22} holds for every $(u_j)_{j=1}^N \subset (0,\infty)$,
$(g_j)_{j=1}^N \subset L^p((0,\infty)^m)$ and $N\in\N$.

\begin{Lem}\label{Lem p.18}
    Let $\alpha \in(-1/2,\infty)^m$. The family of operators
    $\{R_{\alpha, 1}(u)\}_{u>0}$ is $R$-bounded in \newline $L^p((0,\infty)^m)$.
\end{Lem}

\begin{proof}
    It is enough to take into account the above observation and \eqref{eq2.2}, \eqref{A1}, \eqref{A2} and \eqref{A3}.
\end{proof}
The proof of the differentiability of the map
\begin{align*}
        (0,\infty) & \longrightarrow \mathcal{L}(L^p((0,\infty)^m)) \\
         u & \longmapsto R_{\alpha, 1}(u)
    \end{align*}
and the $R$-boundedness of the family $\{u \frac{d}{du}R_{\alpha, 1}(u)\}_{u>0}$ are more involved. In order to show these last properties we will use some results established in \cite{JST}.

\begin{Lem}
Let $\alpha \in [1/2,\infty)^m$. The function
  \begin{align*}
        R_{\alpha, 1} : (0,\infty) & \longrightarrow \mathcal{L}(L^p((0,\infty)^m)) \\
        a & \longmapsto R_{\alpha, 1}(a)
    \end{align*}
    is differentiable.
\end{Lem}

\begin{proof}

    In \cite{JST} it was shown that the function
    \begin{align*}
        R_1: (0,\infty) &\longrightarrow \mathcal{L}(L^p(\R^m)) \\
         a&\longmapsto R_1(a)
    \end{align*}
    is differentiable. By identifying $g\in L^p((0,\infty)^m)$ with $g_0\in L^p(\R^m)$ defined by
    $$g_0(x)=\left\{
        \begin{array}{ll}
            g(x), & x_j>0,\;j=1,...,m \\
            0,    & x_j\leq 0,\;\mbox{for some}\;j=1,...,m,
        \end{array}\right.$$
    we have that
    \begin{align*}
        R_1 :  (0,\infty) & \longrightarrow \mathcal{L}(L^p((0,\infty)^m)) \\
        a & \longmapsto R_1(a)
    \end{align*}
    is differentiable.
    In order to show that $R_{\alpha, 1}$ is differentiable we will prove that $D_{\alpha, 1}=R_{\alpha, 1}-R_1$ is differentiable.
    We divide this proof in three steps.\\

    \noindent \textbf{Step 1}: In this first step we show that the operator $D_{\alpha, 1}(a)$ is bounded from $L^p((0,\infty)^m)$ into itself,
    for every $a>0$.\\

    Let $a>0$ and $g\in L^p((0,\infty)^m)$.
    Recall that
    $$G_t^{\alpha_1}(x_1,y_1)=A_{\alpha_1}(1)W_t^{\alpha_1}(x_1,y_1),\;\;\;x_1,y_1,t\in(0,\infty),$$
    and
     $$G_t(x_1,y_1)=A(1)W_t(x_1,y_1),\;\;\;x_1,y_1,t\in(0,\infty).$$
    We can write
    \begin{align*}
         D_{\alpha, 1}(x,y;a)
         =&R_{\alpha, 1}(x,y;a)- R_1(x,y;a)
            =  a^{m/2}[R_{\alpha, 1}(\sqrt{a}x,\sqrt{a}y;1)-R_1(\sqrt{a}x,\sqrt{a}y;1)] \\
          = &{{a^{m/2}}\over{\sqrt{\pi}}}\int_0^{\infty}{[(A_{\alpha_1}(1)\mathcal{W}_t^{\alpha})(\sqrt{a}x,\sqrt{a}y)-(A_1(1)\mathcal{W}_t)(\sqrt{a}x,\sqrt{a}y)]{{dt}\over{\sqrt{t}}}} \\
          = &{{a^{m/2}}\over{\sqrt{\pi}}}\left\{\int_0^{\infty}{[G_t^{\alpha_1}-G_t](\sqrt{a}x_1,\sqrt{a}y_1)\prod_{j=2}^m{W_t^{\alpha_j}(\sqrt{a}x_j,\sqrt{a}y_j)}{{dt}\over{\sqrt{t}}}} \right. \\
        & + \sum_{i=2}^m \int_0^{\infty}{G_t(\sqrt{a}x_1,\sqrt{a}y_1)\prod_{j=2}^{i-1}{W_t(\sqrt{a}x_j,\sqrt{a}y_j)}}  \\
        & \times \left.[W_t^{\alpha_i}(\sqrt{a}x_i,\sqrt{a}y_i)-W_t(\sqrt{a}x_i,\sqrt{a}y_i)]\prod_{j=i+1}^m{W_t^{\alpha_j}(\sqrt{a}x_j,\sqrt{a}y_j)}{{dt}\over{\sqrt{t}}}\right\} \\
          =& \sum_{i=1}^m{M_i(x,y;a)}, \quad x , y \in (0,\infty)^m.
    \end{align*}
    Note that we have apply the identity,
    $$\prod_{j=1}^\ell a_j - \prod_{j=1}^\ell b_j
        = \sum_{i=1}^\ell \Big(\prod_{j=1}^{i-1} b_j \Big) [a_i-b_i] \Big(\prod_{j=i+1}^{m} a_j\Big), \quad a_j, b_j \in \mathbb{R}, \ \ell \in \mathbb{N},$$
    which allows us to compare the terms of each product one by one.

    We now prove that
    \begin{equation}\label{14.1}
        \int_{(0,\infty)^m}{|D_{\alpha, 1}(x,y;a)||g(y)|dy}<\infty,\;\;\;\;x\in(0,\infty)^m.
    \end{equation}

    According to Lemma~\ref{Lem G}, $(b)$, we have that
    \begin{align*}
        & a^{1/2}\int_0^{\infty}{[G_t^{\alpha_1}-G_t](\sqrt{a}x_1,\sqrt{a}y_1){{dt}\over{\sqrt{t}}}} \\
        & \qquad \qquad \leq C\left\{\begin{array}{l}
        \displaystyle{1\over{x_1}}, \quad 0<y_1 \leq{{x_1}\over 2}<\infty, \\
        \\
        \displaystyle{1\over{x_1}}\left(1+\sqrt{{x_1}\over{|x_1-y_1|}}\right), \quad 0<{{x_1}\over 2}<y_1<2x_1<\infty, \\
        \\
        \displaystyle{1\over{y_1}}, \quad 0<2x_1 \leq y_1<\infty,
        \end{array}\right.
    \end{align*}
    where $C>0$ does not depend on $a$.

    From now on, to abbreviate notation, if $x=(x_1, \dots, x_m) \in (0,\infty)^m$,
    we simply call $\bar{x}_j=(x_1, \dots, x_{j-1},x_{j+1}, \dots, x_m) \in (0,\infty)^{m-1}$, $j =1, \dots, m$.
    Hence, from Lemma~\ref{Lem W}, $(a)$, we get
    \begin{align*}
        &\int_{(0,\infty)^m} {|M_1(x,y;a)||g(y)|dy}
            \leq C \left[\int_0^{x_1/2}{{1\over{x_1}}}+\int_{x_1/2}^{2x_1}{{1\over{x_1}}\left(1+\sqrt{{x_1}\over{|x_1-y_1|}}\right)} +\int_{2x_1}^{\infty}{{1\over{y_1}}}\right] \\
        & \qquad \qquad \times \Big\{\sup_{t>0}\int_{(0,\infty)^{m-1}}{a^{(m-1)/ 2} \frac{e^{-ca|\bar{x}_1-\bar{y}_1|^2/t}}{t^{(m-1)/2}}|g(y)|}d\bar{y}_1 \Big\}dy_1.
     \end{align*}
    We consider the maximal operator
    $$\frak{W}_*(h)(x)
        =\sup_{t>0}\left|\int_{(0,\infty)^{m-1}} \frac{e^{-ca|x-y|^2/t}}{t^{(m-1)/2}} h(y)dy\right|,\;\;\;x\in(0,\infty)^{m-1},$$
    which is bounded from $L^p((0,\infty)^{m-1})$ into itself for each $1<p<\infty$.

     Moreover, the Hardy operators \cite[p. 244, (9.9.1) and (9.9.2)]{HLP}
     $$H_0(F)(x)
        ={1\over x}\int_0^x{F(y)dy}\;\;\;\mbox{and}\;\;\;H^{\infty}F(x)=\int_x^{\infty}{{F(y)\over y}dy},\;\;\;x\in(0,\infty),$$
     and the operator
     $$N(F)(x)
        ={1\over x}\int_{x/2}^{2x}{\left(1+\sqrt{x\over{|x-y|}}\right)F(y)dy},\;\;\;x\in (0,\infty),$$
     are bounded from $L^p(0,\infty)$ into itself for each $1<p<\infty$. Then
     $$\int_{(0,\infty)^m} {|M_1(x,y;a)||g(y)|dy}<\infty,\;\;\;\mbox{a.e.}\;x\in(0,\infty)^m.$$

     Next we study $M_2$. From Lemma~\ref{Lem W}, $(b)$, we obtain
    \begin{align*}
        & \int_0^{\infty}\left|W_t^{\alpha_2}(\sqrt{a}x_2,\sqrt{a}y_2)-W_t(\sqrt{a}x_2,\sqrt{a}y_2)\right|{{dt}\over t}
            \leq C\left(\int_0^{x_2y_2a}{1\over{\sqrt{t}x_2y_2a}}dt + \int_{x_2y_2a}^{\infty}{1\over{t^{3/2}}}dt\right)\\
        & \qquad \leq {C\over{\sqrt{x_2y_2a}}},\;\;\;x_2,y_2,a\in(0,\infty),
    \end{align*}
and
\begin{align*}
        & \int_0^{\infty}\left|W_t^{\alpha_2}(\sqrt{a}x_2,\sqrt{a}y_2)-W_t(\sqrt{a}x_2,\sqrt{a}y_2)\right|{{dt}\over t}
            \leq C\int_0^{\infty}{{e^{-ca|x_2-y_2|^2/t}}\over{t^{3/2}}}dt \\
        & \qquad \leq {C\over{a^{1/2}|x_2-y_2|}}
            \leq C\left\{\begin{array}{l}
                \displaystyle {1\over{\sqrt{a}x_2}},\;\;0<y_2<x_2/2 \\
                \\
                \displaystyle {1\over{\sqrt{a}y_2}},\;\;2x_2<y_2<\infty.
         \end{array}\right.
            \end{align*}

  The above estimates allow us to write, for every $x\in(0,\infty)^m$,
 \begin{align*}
    & \int_{(0,\infty)^m}|M_2(x,y;a)||g(y)|dy \\
    & \qquad \qquad \leq Ca^{(m-1)/ 2}\left(\int_0^{x_2/2}{{1\over{x_2}}}+\int_{x_2/2}^{2x_2}{{1\over{x_2}}} +\int_{2x_2}^{\infty}{{1\over{y_2}}}\right)
         \Big\{\sup_{t>0}\int_{(0,\infty)^{m-1}} {e^{-ca|\bar{x}_2-\bar{y}_2|^2/t}\over{t^{(m-1)/2}}}|g(y)|d\bar{y}_2 \Big\} dy_2.
 \end{align*}
         where $C$ is not depending on $a$.

So, as in the case of $M_1$, we conclude
$$\int_{(0,\infty)^m}|M_2(x,y;a)||g(y)|dy < \infty,\;\;\;\mbox{a.e.}\;x\in(0,\infty)^m.$$

In a similar way we can see that, for every $j=3,...,m$,
$$\int_{(0,\infty)^m}|M_j(x,y;a)||g(y)|dy <\infty,\;\;\;\mbox{a.e.}\;x\in(0,\infty)^m.$$
Hence, \eqref{14.1} is justified. Thus, we can write
    $$D_{\alpha,1}(a)(g)(x)
        =\int_{(0,\infty)^m}{[R_{\alpha, 1}(x,y;a)-R_1(x,y;a)]g(y)dy}, \quad \mbox{a.e.}\;x\in (0,\infty)^m,$$
where the last integral is absolutely convergent. Moreover, by using the $L^p$-boundedness of the auxiliar operators we conclude that
\begin{equation}\label{eq2.8.1.F1}
      \left\|D_{\alpha ,1}(a)(g)\right\|_{L^p((0,\infty)^m)}\leq C\|g\|_{L^p((0,\infty)^m)},\;\;\;g\in L^p((0,\infty)^m),
    \end{equation}
    where $C>0$ does not depend on $a$.\\

    \noindent \textbf{Step 2}: The aim of this step is to show that
    \begin{equation*}
        \|d_{\alpha,1}(a)(g)\|_{L^p((0,\infty)^m)}
            \leq \frac{C}{a} \|g\|_{L^p((0,\infty)^m)}, \quad a>0,
    \end{equation*}
    being $d_{\alpha,1}(a)$ the operator given by
    $$d_{\alpha,1}(a)(g)(x)
        = \int_{(0,\infty)^m}  \partial_a [ R_{\alpha,1}(x,y;a) - R_1(x,y;a)]g(y)dy.$$ \\

    We now analyze ${\partial}_a[R_{\alpha, 1}(x,y;a)-R_1(x,y;a)]$, $x,y\in (0,\infty)^m$.
    We have that
    \begin{align}\label{W1}
        & {\partial}_a[R_{\alpha, 1}(x,y;a)-R_1(x,y;a)]=
         \partial_a\left(a^{m/2}[R_{\alpha, 1}(\sqrt{a}x,\sqrt{a}y;1)-R_1(\sqrt{a}x,\sqrt{a}y;1)]\right) \nonumber \\
        & \qquad ={m\over 2}a^{(m-2)/2}[R_{\alpha, 1}(\sqrt{a}x,\sqrt{a}y;1)-R_1(\sqrt{a}x,\sqrt{a}y;1)] \nonumber \\
        & \qquad \qquad +{{a^{(m-1)/2}}\over 2}\int_0^{\infty}\left(x_1(\partial_{x_1}G_t^{\alpha_1})(\sqrt{a}x_1,\sqrt{a}y_1)\prod_{j=2}^mW_t^{\alpha_j}(\sqrt{a}x_j,\sqrt{a}y_j) \right. \nonumber \\
        & \qquad \qquad + G_t^{\alpha_1}(\sqrt{a}x_1,\sqrt{a}y_1)\sum_{j=2}^mx_j(\partial_{x_j}W_t^{\alpha_j})(\sqrt{a}x_j,\sqrt{a}y_j)\prod_{i=2,i\neq j}^mW_t^{\alpha_i}(\sqrt{a}x_i,\sqrt{a}y_i) \nonumber \\
        & \qquad \qquad +y_1(\partial_{y_1}G_t^{\alpha_1})(\sqrt{a}x_1,\sqrt{a}y_1)\prod_{j=2}^mW_t^{\alpha_j}(\sqrt{a}x_j,\sqrt{a}y_j) \nonumber \\
        & \qquad \qquad +G_t^{\alpha_1}(\sqrt{a}x_1,\sqrt{a}y_1)\sum_{j=2}^my_j(\partial_{y_j}W_t^{\alpha_j})(\sqrt{a}x_j,\sqrt{a}y_j)\prod_{i=2,i\neq j}^mW_t^{\alpha_i}(\sqrt{a}x_i,\sqrt{a}y_i) \nonumber \\
        & \qquad \qquad - x_1(\partial_{x_1}G_t)(\sqrt{a}x_1,\sqrt{a}y_1)\prod_{j=2}^mW_t(\sqrt{a}x_j,\sqrt{a}y_j) \nonumber \\
        & \qquad \qquad - G_t(\sqrt{a}x_1,\sqrt{a}y_1)\sum_{j=2}^mx_j(\partial_{x_j}W_t)(\sqrt{a}x_j,\sqrt{a}y_j)\prod_{i=2,i\neq j}^mW_t(\sqrt{a}x_i,\sqrt{a}y_i) \nonumber \\
        & \qquad \qquad -y_1(\partial_{y_1}G_t)(\sqrt{a}x_1,\sqrt{a}y_1)\prod_{j=2}^mW_t(\sqrt{a}x_j,\sqrt{a}y_j) \nonumber \\
        & \qquad \qquad \left. -G_t(\sqrt{a}x_1,\sqrt{a}y_1)\sum_{j=2}^my_j(\partial_{y_j}W_t)(\sqrt{a}x_j,\sqrt{a}y_j)\prod_{i=2,i\neq j}^mW_t(\sqrt{a}x_i,\sqrt{a}y_i)\right){{dt}\over{\sqrt{\pi t}}}
    \end{align}

    According to \eqref{eq2.8.1.F1} we get, for every $g\in L^p((0,\infty)^m)$,
  \begin{align}\label{eq2.8.1.F2}
      & \left\|a^{(m-2)/2}\int_{(0,\infty)^m}[R_{\alpha, 1}(\sqrt{a}x,\sqrt{a}y;1)-R_1(\sqrt{a}x,\sqrt{a}y;1)]g(y)dy\right\|_{L^p((0,\infty)^m)} \nonumber \\
      & \qquad \qquad \leq {C\over a}\|g\|_{L^p((0,\infty)^m)}.
    \end{align}

    We want to obtain analogous $L^p$-estimates for the remaining terms in \eqref{W1}. We are going to reorganize them in the new kernels
    $S_t$ and $T_t$ defined below (see \eqref{19.1} and \eqref{25.1}).\\

    We consider, for every $x, y \in(0,\infty)^m$ and $t>0$, the kernel $S_t$ given by
   \begin{align}\label{19.1}
       & S_t(x,y;a)
           =  G_t^{\alpha_1}(\sqrt{a}x_1,\sqrt{a}y_1) \sum_{j=2}^m (x_j\partial_{x_j}+y_j\partial_{y_j})W_t^{\alpha_j}(\sqrt{a}x_j,\sqrt{a}y_j) \prod_{i=2,i\neq j}^mW_t^{\alpha_i}(\sqrt{a}x_i,\sqrt{a}y_i)\nonumber \\
       & \quad \quad -  G_t(\sqrt{a}x_1,\sqrt{a}y_1)\sum_{j=2}^m(x_j\partial_{x_j}+y_j\partial_{y_j})W_t(\sqrt{a}x_j,\sqrt{a}y_j) \prod_{i=2,i\neq j}^mW_t(\sqrt{a}x_i,\sqrt{a}y_i) \nonumber \\
       & \quad =  \Big[G_t^{\alpha_1}(\sqrt{a}x_1,\sqrt{a}y_1)-G_t(\sqrt{a}x_1,\sqrt{a}y_1)\Big] \sum_{j=2}^m  (x_j\partial_{x_j}+y_j\partial_{y_j})W_t^{\alpha_j}(\sqrt{a}x_j,\sqrt{a}y_j) \nonumber \\
       & \qquad \qquad  \times \prod_{i=2,i\neq j}^mW_t^{\alpha_i}(\sqrt{a}x_i,\sqrt{a}y_i) \nonumber \\
       & \quad \quad + G_t(\sqrt{a}x_1,\sqrt{a}y_1)\sum_{j=2}^m(x_j\partial_{x_j}+y_j\partial_{y_j})\Big[W_t^{\alpha_j}(\sqrt{a}x_j,\sqrt{a}y_j)- W_t(\sqrt{a}x_j,\sqrt{a}y_j)\Big] \nonumber \\
       & \qquad \qquad \times \prod_{i=2,i\neq j}^mW_t^{\alpha_i}(\sqrt{a}x_i,\sqrt{a}y_i) \nonumber \\
       & \quad \quad + G_t(\sqrt{a}x_1,\sqrt{a}y_1)\sum_{j=2}^m(x_j\partial_{x_j}+y_j\partial_{y_j})W_t(\sqrt{a}x_j,\sqrt{a}y_j) \nonumber \\
       & \qquad \qquad  \times \sum_{k=2,k\neq j}^m \Big[W_t^{\alpha_k}(\sqrt{a}x_k,\sqrt{a}y_k)- W_t(\sqrt{a}x_k,\sqrt{a}y_k) \Big] \prod_{s=2,s\neq j}^{k-1}W_t(\sqrt{a}x_s,\sqrt{a}y_s)  \nonumber \\
       & \qquad \qquad  \times  \prod_{s=k+1,s\neq j}^{m}W_t^{\alpha_s}(\sqrt{a}x_s,\sqrt{a}y_s) \nonumber \\
       & \quad = S_t^1(x,y;a) + S_t^2(x,y;a) + S_t^3(x,y;a)
    \end{align}

    We study each of the above terms separately. First of all we analyze $S_t^1$. Lemma~\ref{Lem G}, $(b)$, with $\eps=1/4$, implies that
         \begin{align}\label{eq2.8.1.Z9}
            & \int_0^{\infty}\left|G_t^{\alpha_1}(\sqrt{a}x_1,\sqrt{a}y_1)-G_t(\sqrt{a}x_1,\sqrt{a}y_1)\right|e^{5t/4}{{dt}\over{\sqrt{t}}} \nonumber \\
            & \qquad \qquad \leq \frac{C}{\sqrt{a x_1}} \int_0^{\infty}{{e^{-ca|x_1-y_1|^2/t}}\over{t^{5/4}}}dt
                            \leq {{C}\over{\sqrt{a}}}{1\over{x_1}}\sqrt{{x_1}\over{|x_1-y_1|}}, \quad 0<x_1/2<y_1<2x_1,
        \end{align}
    and
        \begin{align}\label{eq2.8.1.Z10}
            & \int_0^{\infty}\left|G_t^{\alpha_1}(\sqrt{a}x_1,\sqrt{a}y_1)-G_t(\sqrt{a}x_1,\sqrt{a}y_1)\right|e^{5t/4}{{dt}\over{\sqrt{t}}} \nonumber \\
            & \qquad \leq C\sqrt{a}\max\{x_1,y_1\}\int_0^{\infty}{{e^{-ca|x_1-y_1|^2/t}}\over{t^2}}dt
            \leq {{C}\over{\sqrt{a}}} {{\max\{x_1,y_1\}}\over{|x_1-y_1|^2}}
             \leq {{C}\over{\sqrt{a}}}
            \left\{\begin{array}{l}
                \displaystyle {1\over{x_1}},\;\;0<y_1<x_1/2, \\
                \\
                \displaystyle {1\over{y_1}},\;\;0<2x_1<y_1.
            \end{array}\right.
        \end{align}

        By combining Lemma~\ref{Lem W}, $(a)$ and $(e)$ taken with $\eps=1/4$;  \eqref{eq2.8.1.Z9} and  \eqref{eq2.8.1.Z10} we deduce that,
        for every $g\in L^p((0,\infty)^m)$, $1<p<\infty$ and $j=2,...,m$,
        \begin{align*}
            & \Big| \int_{(0,\infty)^m}g(y)\int_0^\infty \Big[ G_t^{\alpha_1}(\sqrt{a}x_1,\sqrt{a}y_1)-G_t(\sqrt{a}x_1,\sqrt{a}y_1)\Big]
                    \prod_{i=2,i\neq j}^mW_t^{\alpha_i}(\sqrt{a}x_i,\sqrt{a}y_i) \\
            & \qquad  \times \Big[ x_j(\partial_{x_j}W_t^{\alpha_j})(\sqrt{a}x_j,\sqrt{a}y_j)+y_j(\partial_{y_j}W_t^{\alpha_j})(\sqrt{a}x_j,\sqrt{a}y_j) \Big] {{dt}\over{\sqrt{t}}}dy\Big| \\
            & \quad \leq {C\over a}\Big(\int_0^{x_1/2}{{1\over{x_1}}}+\int_{x_1/2}^{2x_1}{{1\over{x_1}}\sqrt{{x_1}\over{|x_1-y_1|}}}    +\int_{2x_1}^{\infty}{{1\over{y_1}}}\Big)
              \Big\{\sup_{t>0}\int_{(0,\infty)^{m-1}} {{e^{-ca|\bar{x}_1-\bar{y}_1|^2/t}}\over{t^{(m-1)/2}}} |g(y)|d\bar{y}_1\Big\}dy_1,
        \end{align*}
     and the $L^p$-boundedness property for the Hardy and maximal operators leads to
     \begin{align}\label{eq2.8.1.Z11}
            & a^{(m-1)/ 2}\left\|\int_{(0,\infty)^m}g(y)\int_0^\infty S_t^1(x,y;a){{dt}\over{\sqrt{t}}}dy\right\|_{L^p((0,\infty)^m)}
                 \leq {C\over a}\|g\|_{L^p((0,\infty)^m)},
     \end{align}
     where $C>0$ does not depend on $a$.\\

     Now we pass to the study of the kernel $S_t^3$.  Let $k=2,...,m$. Lemma~\ref{Lem W}, $(b)$, allows us to write
        \begin{align}\label{eq2.8.1.Z14}
            & \int_0^{\infty} \left|W_t^{\alpha_k}(\sqrt{a}x_k,\sqrt{a}y_k)-W_t(\sqrt{a}x_k,\sqrt{a}y_k)\right|{{dt}\over t} \nonumber\\
            & \qquad \qquad \leq C\int_0^{\infty} {{e^{-t/2}}\over{(\fa a x_k y_k)^{1/4}}} {{e^{-ca|x_k-y_k|^2/t}}\over{t^{3/2}}}dt
                \leq {C\over{(ax_ky_k)^{1/4}}}\int_0^{\infty}{{e^{-ca|x_k-y_k|^2/t}}\over{t^{5/4}}}dt \nonumber \\
            & \qquad \qquad \leq \frac{C}{\sqrt{a}} \frac{1}{x_k} \sqrt{\frac{x_k}{|x_k-y_k|}}, \quad x_k/2 < y_k < 2x_k,
        \end{align}
        and from Lemma~\ref{Lem W}, $(a)$, we obtain
         \begin{align}\label{eq2.8.1.Z15}
            & \int_0^{\infty}\left|W_t^{\alpha_k}(\sqrt{a}x_k,\sqrt{a}y_k)-W_t(\sqrt{a}x_k,\sqrt{a}y_k)\right|{{dt}\over t} \nonumber \\
            & \qquad \qquad \leq C\int_0^{\infty}{{e^{-ca|x_k-y_k|^2/t}}\over{t^{3/2}}}dt
                    \leq \frac{C}{\sqrt{a}} {1\over{|x_k-y_k|}}\leq{C\over{\sqrt{a}}}
            \left\{\begin{array}{l}
                \displaystyle {1\over{x_k}},\;\;0<y_k<x_k/2, \\
                \\
                \displaystyle {1\over{y_k}},\;\;0<2x_k<y_k.
            \end{array}\right.
        \end{align}

        By combining Lemma~\ref{Lem W}, $(a)$ and $(d)$; Lemma~\ref{Lem G}, $(a)$;  \eqref{eq2.8.1.Z14} and \eqref{eq2.8.1.Z15}  we deduce
        \begin{align}\label{19.2}
            & \Big|\int_{(0,\infty)^m}g(y)\int_0^\infty G_t(\sqrt{a}x_1,\sqrt{a}y_1)\sum_{j=2}^m\left(x_j\partial_{x_j}+y_j\partial_{y_j}\right)W_t(\sqrt{a}x_j,\sqrt{a}y_j) \nonumber \\
            & \qquad \qquad \times \sum_{k=2,k\neq j}^m \Big[W_t^{\alpha_k}(\sqrt{a}x_k,\sqrt{a}y_k)-W_t(\sqrt{a}x_k,\sqrt{a}y_k)\Big]   \prod_{s=2,s\neq j}^{k-1}W_t(\sqrt{a}x_s,\sqrt{a}y_s) \nonumber\\
            &\qquad \qquad  \times \prod_{s=k+1,s\neq j}^mW_t^{\alpha_s}(\sqrt{a}x_s,\sqrt{a}y_s) {{dt}\over{\sqrt{t}}}dy \Big| \nonumber \\
            & \qquad \leq {C\over a}\sum_{j=2}^m\sum_{k=2,k\neq j}^m\left(\int_0^{x_k/2}{{1\over{x_k}}}+\int_{x_k/2}^{2x_k}{{1\over{x_k}}\sqrt{{x_k}\over{|x_k-y_k|}}}
            +\int_{2x_k}^{\infty}{{1\over{y_k}}}\right) \nonumber \\
            & \qquad \qquad \times \Big\{ \sup_{t>0}\int_{(0,\infty)^{m-1}} {{e^{-ca|\bar{x}_k-\bar{y}_k|^2/t}}\over{t^{(m-1)/2}}}|g(y)|d\bar{y}_k \Big\} dy_k, \quad x \in(0,\infty)^m.
        \end{align}

     Then, for every $g\in L^p((0,\infty)^m)$,
     \begin{align}\label{eq2.8.1.Z17}
       & a^{(m-1)/ 2}\left\|\int_{(0,\infty)^m}g(y)\int_0^\infty S_t^3(x,y;a){{dt}\over{\sqrt{t}}}dy\right\|_{L^p((0,\infty)^m)}
                \leq {C\over a}\|g\|_{L^p((0,\infty)^m)},
     \end{align}
     where $C>0$ does not depend on $a$.\\

     Next we concentrate on the kernel $S_t^2$. Lemma~\ref{Lem W}, $(f)$, give us
    \begin{align}\label{eq2.8.1.Z24}
        & \int_0^{\infty}\Big|x_j\left[\partial_{x_j}(W_t^{\alpha_j}-W_t)\right](\sqrt{a}x_j,\sqrt{a}y_j)+y_j\left[\partial_{y_j}(W_t^{\alpha_j}-W_t)\right](\sqrt{a}x_j,\sqrt{a}y_j)\Big|{{dt}\over t} \nonumber \\
        & \qquad \qquad \leq {C \over a} {1\over{(x_jy_j)^{1/4}|x_j-y_j|^{1/2}}},\;\;\;x_j,y_j\in(0,\infty).
     \end{align}
     Also, from Lemma~\ref{Lem W}, $(d)$ and $(e)$, we get
      \begin{align}\label{eq2.8.1.Z25}
      & \int_0^{\infty}\Big|x_j\left[\partial_{x_j}(W_t^{\alpha_j}-W_t)\right](\sqrt{a}x_j,\sqrt{a}y_j)+y_j\left[\partial_{y_j}(W_t^{\alpha_j}-W_t)\right](\sqrt{a}x_j,\sqrt{a}y_j)\Big|{{dt}\over t} \nonumber \\
      & \qquad \qquad \leq {C \over a} {1\over{|x_j-y_j|}}\leq {{C}\over{\sqrt{a}}}\left\{
      \begin{array}{l}
        \displaystyle {1\over{x_j}},\;\;0<y_j<x_j/2, \\
         \\
        \displaystyle {1\over{y_j}},\;\;0<2x_j<y_j.
        \end{array}\right.
     \end{align}

     By combining Lemma~\ref{Lem W}, $(a)$; Lemma~\ref{Lem G}, $(a)$; \eqref{eq2.8.1.Z24} and \eqref{eq2.8.1.Z25} we obtain, for every $g\in L^p((0,\infty)^m)$, $1<p<\infty$,
     \begin{align*}
        & \left| \int_{(0,\infty)^m}g(y)\int_0^{\infty}G_t(\sqrt{a}x_1,\sqrt{a}y_1)\sum_{j=2}^m
                (x_j\partial_{x_j}+y_j\partial_{y_j})\Big[W_t^{\alpha_j}(\sqrt{a}x_j,\sqrt{a}y_j)-W_t(\sqrt{a}x_j,\sqrt{a}y_j)\Big] \right. \\
        & \quad \quad \left.\times\prod_{i=2,i\neq j}^m W_t^{\alpha_i}(\sqrt{a}x_i,\sqrt{a}y_i){{dt}\over{\sqrt{t}}}dy\right| \\
        & \quad \leq {C\over a}\sum_{j=2}^m\left(\int_0^{x_j/2}{{1\over{x_j}}}+\int_{x_j/2}^{2x_j}{{1\over{x_j}}\sqrt{{x_j}\over{|x_j-y_j|}}}
                +\int_{2x_j}^{\infty}{{1\over{y_j}}}\right)
         \Big\{\sup_{t>0}\int_{(0,\infty)^{m-1}} {{e^{-ca|\bar{x}_j-\bar{y}_j|^2/t}}\over{t^{(m-1)/2}}}|g(y)|d\bar{y}_j \Big\} dy_j.
    \end{align*}

    $L^p$-boundedness properties of the Hardy and maximal operators lead to
    \begin{align}\label{eq2.8.1.Z27}
        a^{(m-1)/ 2} \left\|\int_{(0,\infty)^m} g(y)\int_0^\infty S_t^2(x,y;a){{dt}\over{\sqrt{t}}}dy\right\|_{L^p((0,\infty)^m)}
            \leq {C\over a}\|g\|_{L^p((0,\infty)^m)},
     \end{align}
     for every $g\in L^p((0,\infty)^m)$, $1<p<\infty$. Here $C>0$ is not depending on $a$.\\

     Putting together \eqref{eq2.8.1.Z11}, \eqref{eq2.8.1.Z17} and \eqref{eq2.8.1.Z27} we obtain
      \begin{align}\label{eq2.8.1.Z28}
       & a^{(m-1)/ 2}\left\|\int_{(0,\infty)^m}g(y)\int_0^\infty S_t(x,y;a){{dt}\over{\sqrt{t}}}dy\right\|_{L^p((0,\infty)^m)} \leq {C\over a}\|g\|_{L^p((0,\infty)^m)},
     \end{align}
     for every $g\in L^p((0,\infty)^m)$, $1<p<\infty$, where $C>0$ does not depend on $a$.\\

     We now define the kernel
     \begin{align}\label{25.1}
        T_t(x,y;a)
        =  &    \Big(x_1\left[\partial_{x_1}(G_t^{\alpha_1}-G_t)\right](\sqrt{a}x_1,\sqrt{a}y_1)+y_1\left[\partial_{y_1}(G_t^{\alpha_1}-G_t)\right](\sqrt{a}x_1,\sqrt{a}y_1)\Big) \nonumber \\
        & \qquad \times\prod_{j=2}^mW_t^{\alpha_j}(\sqrt{a}x_j,\sqrt{a}y_j) \nonumber \\
        & + \Big( x_1 \partial_{x_1} G_t(\sqrt{a}x_1,\sqrt{a}y_1) + y_1 \partial_{y_1} G_t(\sqrt{a}x_1,\sqrt{a}y_1) \Big) \nonumber \\
        & \qquad \times \Big( \prod_{j=2}^m W_t^{\alpha_j}(\sqrt{a}x_j,\sqrt{a}y_j) - \prod_{j=2}^m W_t(\sqrt{a}x_j,\sqrt{a}y_j) \Big) \nonumber \\
        =&   T_t^1(x,y;a) + T_t^2(x,y;a), \quad x, y \in(0,\infty)^m\;\mbox{and}\;t>0.
    \end{align}

    We start with $T_t^1$.  By Lemma~\ref{Lem G}, $(d)$ and $(e)$, it follows that
    \begin{align}\label{eq2.8.1.Z35}
        & \int_0^\infty \Big| x_1 \partial_{x_1}  \Big( G_t - G_t^{\alpha_1} \Big)(\sqrt{a}x_1,\sqrt{a}y_1) + y_1 \partial_{y_1}  \Big( G_t - G_t^{\alpha_1}\Big)(\sqrt{a}x_1,\sqrt{a}y_1) \Big| \frac{dt}{\sqrt{t}} \nonumber \\
        & \quad \leq \frac{C}{a} \Big[ \frac{1}{|x_1-y_1|^{1/2}} \Big( \frac{x_1^{1/4}}{y_1^{3/4}} + \frac{y_1^{1/4}}{x_1^{3/4}} \Big)
                  +   \frac{1}{\sqrt{x_1^2+y_1^2}}\Big]
         \leq \frac{C}{a} \frac{1}{x_1} \Big( 1 + \sqrt{\frac{x_1}{|x_1-y_1|}} \Big), \quad  x_1/2 < y_1 < 2 x_1,
    \end{align}
    and from  Lemma~\ref{Lem G}, $(f)$, we deduce that
     \begin{align}\label{eq2.8.1.Z36}
        & \int_0^\infty \Big| x_1 \partial_{x_1}  \Big( G_t - G_t^{\alpha_1} \Big)(\sqrt{a}x_1,\sqrt{a}y_1) + y_1 \partial_{y_1}  \Big( G_t - G_t^{\alpha_1}\Big)(\sqrt{a}x_1,\sqrt{a}y_1) \Big| \frac{dt}{\sqrt{t}} \nonumber \\
        & \qquad \leq \frac{C}{\sqrt{a}} \int_0^{\infty}{{e^{-ca\max\{x_1,y_1\}^2/t}}\over{t^{3/2}}}dt
         \leq  \frac{C}{a}
           \left\{\begin{array}{l}
           \displaystyle{1\over{x_1}},\;\;0<y_1<x_1/2. \\
         \\
         \displaystyle{1\over{y_1}},\;\;0<2x_1<y_1.
          \end{array}\right.
    \end{align}

    According to  Lemma~\ref{Lem W}, $(a)$;  \eqref{eq2.8.1.Z35} and \eqref{eq2.8.1.Z36}
    it follows that, for every $g\in L^p((0,\infty)^m)$, $1<p<\infty$ and $x\in(0,\infty)^m$,
     \begin{align*}
         \left|\int_{(0,\infty)^m}g(y)\int_0^{\infty}T_t^1(x,y,a){{dt}\over{\sqrt{t}}}dy\right|
                \leq & \frac{C}{a}\left(\int_0^{x_1/2}{{1\over{x_1}}}+\int_{x_1/2}^{2x_1}{{1\over{x_1}}\Big(1+\sqrt{{x_1}\over{|x_1-y_1|}}}\Big) +\int_{2x_1}^{\infty}{{1\over{y_1}}}\right) \nonumber \\
        & \times  \Big\{ \sup_{t>0}\int_{(0,\infty)^{m-1}} {{e^{-ca|\bar{x}_1-\bar{y}_1|^2/t}}\over{t^{(m-1)/2}}}|g(y)| d\bar{y}_1 \Big\} dy_1.
    \end{align*}

    And, $L^p$-boundedness properties of Hardy and maximal operators, together with, Jensen inequality, imply that
      \begin{align}\label{eq2.8.1.Z38}
       & a^{(m-1)/ 2}\left\|\int_{(0,\infty)^m}g(y)\int_0^\infty T_t^1(x,y,a){{dt}\over{\sqrt{t}}}dy\right\|_{L^p((0,\infty)^m)} \leq {C\over a}\|g\|_{L^p((0,\infty)^m)},
     \end{align}
     where $C>0$ does not depend on $a$.\\

     Now, we consider the kernel $T_t^2$. First of all, observe that
     \begin{align*}
        T_t^2(x,y;a)
            = & \Big( x_1 \partial_{x_1} G_t(\sqrt{a}x_1,\sqrt{a}y_1) + y_1 \partial_{y_1} G_t(\sqrt{a}x_1,\sqrt{a}y_1) \Big)  \\
              &  \times \sum_{i=2}^m \prod_{j=2}^{i-1} W_t(\sqrt{a}x_j,\sqrt{a}y_j)
                    \Big[ W_t^{\alpha_i}(\sqrt{a}x_i,\sqrt{a}y_i) - W_t(\sqrt{a}x_i,\sqrt{a}y_i) \Big]  \\
              & \times \prod_{j=i+1}^{m} W_t^{\alpha_j}(\sqrt{a}x_j,\sqrt{a}y_j).
     \end{align*}
     By taking in mind Lemma~\ref{Lem W}, $(a)$; Lemma~\ref{Lem G}, $(c)$; \eqref{eq2.8.1.Z14}, \eqref{eq2.8.1.Z15} and
     proceeding as in \eqref{19.2} we conclude, for every $g\in L^p((0,\infty)^m)$,
     \begin{align}\label{Tt2}
       & a^{(m-1)/ 2}\left\|\int_{(0,\infty)^m}g(y)\int_0^\infty T_t^2(x,y;a){{dt}\over{\sqrt{t}}}dy\right\|_{L^p((0,\infty)^m)}
                \leq {C\over a}\|g\|_{L^p((0,\infty)^m)},
     \end{align}
     where $C>0$ does not depend on $a$.

     Hence, \eqref{eq2.8.1.Z38} and \eqref{Tt2} give us
     \begin{align}\label{Tt}
       & a^{(m-1)/ 2}\left\|\int_{(0,\infty)^m}g(y)\int_0^\infty T_t(x,y;a){{dt}\over{\sqrt{t}}}dy\right\|_{L^p((0,\infty)^m)}
                \leq {C\over a}\|g\|_{L^p((0,\infty)^m)},
     \end{align}
     being the constant $C>0$ independent of $a$.

     By combining \eqref{eq2.8.1.F2}, \eqref{eq2.8.1.Z28} and \eqref{Tt} we conclude that, for every $g\in L^p((0,\infty)^m)$, $1<p<\infty$
     \begin{align}\label{eq2.8.1.Z39.0}
       & \left\|\int_{(0,\infty)^m}  \partial_a \Big( R_{\alpha,1}(x,y;a) - R_1(x,y;a)\Big)g(y)dy\right\|_{L^p((0,\infty)^m)} \leq {C\over a}\|g\|_{L^p((0,\infty)^m)},
     \end{align}

     Here, as above, $C>0$ does not depend on $a$. By proceeding in a similar way we can also obtain
      \begin{align}\label{eq2.8.1.Z39}
       & \left\|\int_{(0,\infty)^m}  \partial_a^2 \Big( R_{\alpha,1}(x,y;a) - R_1(x,y;a)\Big)g(y)dy\right\|_{L^p((0,\infty)^m)} \leq {C\over{a^{3/2}}}\|g\|_{L^p((0,\infty)^m)},
     \end{align}
     for every $g\in L^p((0,\infty)^m)$, $1<p<\infty$, where $C>0$ is not depending on $a$.\\

     \noindent \textbf{Step 3}: This last step is devoted to show the differentiability of the operator $D_{\alpha,1}$.\\

     Let  $g\in L^p((0,\infty)^m)$, $1<p<\infty$. We can write
      \begin{align*}
        & \frac{1}{h} \Big[ D_{\alpha,1}(a+h)(g)(x) - D_{\alpha.1}(a)(g)(x)\Big] - d_{\alpha,1}(a)(g)(x) \\
        & \qquad= \int_{(0,\infty)^m} \Big[ \frac{1}{h} \int_a^{a+h} \partial_{\lambda} \Big( R_{\alpha,1}(x,y;\lambda) - R_1(x,y;\lambda)\Big) d\lambda
                        - \partial_{a} \Big( R_{\alpha,1}(x,y;a) - R_1(x,y;a)\Big) \Big] g(y)dy \\
        &  \qquad = \int_{(0,\infty)^m} \frac{1}{h} \int_a^{a+h} \int_a^\lambda  \partial^2_{z} \Big( R_{\alpha,1}(x,y;z) - R_1(x,y;z)\Big) dz d\lambda g(y) dy,
        \quad  0 < |h|<a \ x \in (0,\infty)^m,
    \end{align*}
    being $d_{\alpha,1}(a)$ the operator considered in Step 2.
    According to \eqref{eq2.8.1.Z39} we get
    \begin{align*}
        & \Big\| \frac{D_{\alpha,1}(a+h)(g) - D_{\alpha,1}(a)(g)}{h} - d_{\alpha,1}(a)(g) \Big\|_{L^p((0,\infty)^m)} \\
        & \qquad = \Big\| \frac{1}{h} \int_a^{a+h} \int_a^\lambda \int_{(0,\infty)^m}  \partial^2_{z} \Big( R_{\alpha,1}(x,y;z) - R_1(x,y;z)\Big) g(y) dy dz d\lambda \Big\|_{L^p((0,\infty)^m)} \\
        &  \qquad\leq \Big| \frac{1}{h} \int_a^{a+h} \int_a^\lambda \Big\|\int_{(0,\infty)^m}  \partial^2_{z} \Big( R_{\alpha,1}(x,y;z) - R_1(x,y;z)\Big) g(y) dy\Big\|_{L^p((0,\infty)^m)} dz d\lambda\big|  \\
         &  \qquad \leq C \|g\|_{L^p((0,\infty)^m)}\Big| \frac{1}{h}\int_a^{a+h} \int_a^\lambda \frac{dz d\lambda}{z^{3/2}} \Big|  \\
        & \qquad \leq C \Big| \frac{h-2\sqrt{a^2+ah}+2a}{\sqrt{a}h} \Big| \|g\|_{L^p((0,\infty)^m)}, \quad 0<|h|<a.
    \end{align*}
    Hence,
    $$\lim_{h \to 0} \frac{D_{\alpha,1}(a+h) - D_{\alpha,1}(a)}{h} = d_{\alpha,1}(a),$$
    in the sense of convergence in $\mathcal{L}(L^p((0,\infty)^m))$.

    We can proceed in a similar way when $a<0$. We conclude that for each $1<p<\infty$ the function
    $$\begin{array}{rcl}
        \R \setminus \{0\} & \longrightarrow & \mathcal{L}(L^p((0,\infty)^m)) \\
        a & \longmapsto & R_{\alpha,1}(a)
    \end{array}$$
    is differentiable  and that, for every $g \in L^p((0,\infty)^m)$,
    \begin{equation*}
        \left[{d\over{da}}R_{\alpha,1}(a)\right]g(x)
            = \lim_{\varepsilon \to 0^+} \int_{|x-y| > \varepsilon} \partial_{a} R_{\alpha,1}(x,y;a)g(y) dy, \quad  \text{a.e. } x \in (0,\infty)^m.
    \end{equation*}
    \end{proof}

\begin{Lem}\label{Lem p.11}
    Let $\alpha \in (1/2,\infty)^m$ and $1<p<\infty$. The family of operators $\{a \frac{d}{da}R_{\alpha,1}(a)\}_{a \in \R \setminus \{0\}}$
    is $R$-bounded in $L^p((0,\infty)^m)$.
\end{Lem}
\begin{proof}
    From \eqref{eq2.8.1.Z39.0} and \cite[Theorem 2.1]{JST} we deduce that
    $$\sup_{a \in \R \setminus \{0\}} \| a{d\over{da}} R_{\alpha,1}(a)\|_{\mathcal{L}(L^2((0,\infty)^m))}<\infty.$$
    Our next objective is to see that
    \begin{equation}\label{eqG0}
        \sup_{a \in \R \setminus \{0\}} \Big|a {d\over{da}} \Big( R_{\alpha,1}(x,y;a) - R_1(x,y;a)  \Big) \Big|
            \leq \frac{C}{|x-y|^m}, \quad  x,y \in (0,\infty)^m, x \neq y.
    \end{equation}

    We will have \eqref{W1} in mind. Let $a>0$. In the sequel the constant $C>0$ does not depend on $a$. As in \eqref{A2} and by \cite{ST1} we get
    \begin{equation}\label{eqG1}
       \Big|a^{(m-2)/2}  \Big( R_{\alpha,1}(\sqrt{a}x,\sqrt{a}y;1) - R_1(\sqrt{a}x,\sqrt{a}y;1)  \Big) \Big|
            \leq \frac{C}{a|x-y|^m}, \quad  x,y \in (0,\infty)^m, x \neq y.
    \end{equation}

    Next we analyze the terms associates to $S_t$ and $T_t$ in \eqref{W1} (see \eqref{19.1} and \eqref{25.1}).

    By Lemma~\ref{Lem W}, $(a)$; as in Lemma~\ref{Lem W}, $(e)$; and Lemma~\ref{Lem G}, $(b)$, we obtain
    \begin{align}\label{eqG4}
        & \int_0^{\infty} \Big| S_t^1(x,y;a) \Big|{{dt}\over{\sqrt{t}}} \nonumber \\
        & \quad \leq C\Big( \chi_{\{x_1/2<y_1<2x_1\}}(y_1) \int_0^{\infty}{{e^{-ca|x-y|^2/t}}\over{\sqrt{a}t^{(m+2)/2}}}dt
                                +\max\{x_1,y_1\}\chi_{\{x_1/2<y_1<2x_1\}^c}(y_1)\int_0^{\infty}{{e^{-ca|x-y|^2/t}}\over{t^{2+(m-1)/2}}}dt\Big) \nonumber \\
        & \quad \leq C\Big({1\over{a^{(m+1)/2}|x-y|^m}}\chi_{\{x_1/2<y_1<2x_1\}}(y_1)+{{\max\{x_1,y_1\}}\over{|x-y|^ma^{m/2}}}\int_0^{\infty}{{e^{-ca|x_1-y_1|^2/t}}\over{t^{3/2}}}dt\chi_{\{x_1/2<y_1<2x_1\}^c}(y_1)\Big) \nonumber \\
        &  \quad \leq C\Big({1\over{a^{(m+1)/2}|x-y|^m}}\chi_{\{x_1/2<y_1<2x_1\}}(y_1)+{{\max\{x_1,y_1\}}\over{|x-y|^ma^{m/2}}}{1\over{|x_1-y_1|\sqrt{a}}}\chi_{\{x_1/2<y_1<2x_1\}^c}(y_1)\Big) \nonumber \\
        & \quad \leq {C\over{a^{(m+1)/2}|x-y|^m}},\;\;\;\;x,y\in(0,\infty)^m,\;x\neq y.
    \end{align}
 We have changed the power $3/4$ of $t$ by $1/2$ in Lemma~\ref{Lem G}, $(b)$, when it was needed (see its proof).

        Also as in Lemma~\ref{Lem W}, $(a)$, $(b)$, $(d)$; and Lemma~\ref{Lem G}, $(a)$; we get
     \begin{align}\label{eqG5}
        & \int_0^{\infty}  \left|S_t^3(x,y;a)\right|{{dt}\over{\sqrt{t}}}
          \leq {C\over{\sqrt{a}}}\int_0^{\infty}{{e^{-ca|x-y|^2/t}}\over{t^{(m+2)/2}}}dt
             \leq {C\over{a^{(m+1)/2}|x-y|^m}},\;\;\;\;x,y\in(0,\infty)^m,\;x\neq y.
     \end{align}
In Lemma~\ref{Lem W}, $(b)$, the factor $(\fa a x_k y_k)^{-1}$ has been removed when it was less than $1$.

        We obtain as in Lemma~\ref{Lem W}, $(a)$ and $(f)$, and Lemma~\ref{Lem G}, $(a)$,
         \begin{align}\label{eqG6}
            & \int_0^{\infty}  \left| S_t^2(x,y;a)\right|{{dt}\over{\sqrt{t}}} \leq {C\over{\sqrt{a}}}\int_0^{\infty}{{e^{-ca|x-y|^2/t}}\over{t^{(m+2)/2}}}dt
                \leq {C\over{a^{(m+1)/2}|x-y|^m}},\;\;\;\;x,y\in(0,\infty)^m,\;x\neq y.
        \end{align}

     On the other hand, from Lemma~\ref{Lem W}, $(a)$; and Lemma~\ref{Lem G}, $(d)$, $(e)$ and $(f)$; we deduce that
    \begin{align}\label{eqG11}
         \int_0^{\infty} & \Big| T_t^1(x,y;a) \Big| \frac{dt}{\sqrt{t}}
       \leq \frac{C}{a^{(m+1)/2} |x-y|^m}, \quad  x,y \in (0,\infty)^m,\;\;x\neq y.
    \end{align}
    Moreover, by Lemma~\ref{Lem W}, $(a)$ and $(b)$; and Lemma~\ref{Lem G}, $(c)$,
    \begin{align}\label{eqG13}
         \int_0^{\infty} & \Big| T_t^2(x,y;a) \Big| \frac{dt}{\sqrt{t}}
       \leq \frac{C}{a^{(m+1)/2} |x-y|^m}, \quad  x,y \in (0,\infty)^m,\;\;x\neq y.
    \end{align}

    Putting together \eqref{eqG1},  \eqref{eqG4}, \eqref{eqG5}, \eqref{eqG6}, \eqref{eqG11} and \eqref{eqG13}, we conclude that \eqref{eqG0} holds.

    The following property holds
     \begin{align}\label{eqG12}
        \sup_{a \in \R \setminus \{0\}} & \left\{\Big|\nabla_x\left(a {d\over{da}} \Big( R_{\alpha,1}(x,y;a) - R_1(x,y;a)  \Big) \right)\Big|+\Big|\nabla_y\left(a {d\over{da}} \Big( R_{\alpha,1}(x,y;a) - R_1(x,y;a)  \Big) \right)\Big|\right\} \nonumber \\
      & \leq \frac{C}{|x-y|^{m+1}}, \quad  x,y \in (0,\infty)^m,\;\;x\neq y.
    \end{align}

    We now prove that
     \begin{align}\label{eqG22}
       &  \sup_{a \in \R \setminus \{0\}} \Big|\partial_{x_1}\left(a {d\over{da}} \Big( R_{\alpha,1}(x,y;a) - R_1(x,y;a)  \Big) \right)\Big| \leq \frac{C}{|x-y|^{m+1}}, \quad  x,y \in (0,\infty)^m,\;\;x\neq y.
    \end{align}

    The proof of \eqref{eqG12} can be completed by the interested reader by using the same ideas used in the proof of \eqref{eqG22} but making careful manipulations.

    Let $a>0$. According to \eqref{W1} we have that, for every $x,y \in (0,\infty), \ x \neq y$
     \begin{align*}
        \partial_{x_1}  & \Big[ \partial_a \Big( R_{\alpha,1}(x,y;a) - R_1(x,y;a) \Big) \Big]
            = {m\over 2}a^{(m-2)/2}\partial_{x_1} \Big[ R_{\alpha,1}(\sqrt{a}x,\sqrt{a}y;1) - R_1(\sqrt{a}x,\sqrt{a}y;1) \Big] \nonumber \\
        & +{{a^{(m-1)/2}}\over 2}\int_0^{\infty}\left\{ \Big[ (\partial_{x_1} G_t^{\alpha_1})(\sqrt{a}x_1,\sqrt{a}y_1) - (\partial_{x_1} G_t)(\sqrt{a}x_1,\sqrt{a}y_1)  \Big]\prod_{j=2}^m W_t^{\alpha_j}(\sqrt{a}x_j,\sqrt{a}y_j) \right.\nonumber \\
        & + x_1\sqrt{a} \Big[(\partial_{x_1}^2 G_t^{\alpha_1})(\sqrt{a}x_1,\sqrt{a}y_1) - (\partial_{x_1}^2 G_t)(\sqrt{a}x_1,\sqrt{a}y_1)\Big]\prod_{j=2}^m W_t^{\alpha_j}(\sqrt{a}x_j,\sqrt{a}y_j) \nonumber \\
        & + \sqrt{a}\Big[(\partial_{x_1} G_t^{\alpha_1})(\sqrt{a}x_1,\sqrt{a}y_1) - (\partial_{x_1} G_t)(\sqrt{a}x_1,\sqrt{a}y_1)\Big] \nonumber \\
        & \times\sum_{j=2}^m\Big(x_j(\partial_{x_j}W_t^{\alpha_j})(\sqrt{a}x_j,\sqrt{a}y_j)+y_j(\partial_{y_j}W_t^{\alpha_j})(\sqrt{a}x_j,\sqrt{a}y_j)\Big)\prod_{i=2,i\neq j}^m W_t^{\alpha_i}(\sqrt{a}x_i,\sqrt{a}y_i) \nonumber \\
        & \left.+ \sqrt{a}y_1\Big[ (\partial_{x_1y_1}^2 G_t^{\alpha_1})(\sqrt{a}x_1,\sqrt{a}y_1) - (\partial_{x_1y_1}^2 G_t)(\sqrt{a}x_1,\sqrt{a}y_1)  \Big]\prod_{j=2}^m W_t(\sqrt{a}x_j,\sqrt{a}y_j)\right\}{{dt}\over{\sqrt{t}}}.
    \end{align*}

    According to \eqref{A3} and \cite{ST1} we obtain for every $x,y \in (0,\infty)^m, \ x \neq y$
     \begin{align}\label{eqG24}
        a^{(m-2)/2}\Big|\partial_{x_1}\Big( R_{\alpha,1}(\sqrt{a}x,\sqrt{a}y;1) - R_1(\sqrt{a}x,\sqrt{a}y;1)  \Big) \Big|
            \leq \frac{C}{a|x-y|^{m+1}}.
    \end{align}

    Then, Lemma~\ref{Lem W}, $(a)$, and Lemma~\ref{Lem G}, $(d)$, lead to
    \begin{align}\label{eqG27}
        a^{(m-1)/2}\int_0^{\infty} & \Big| \left(\partial_{x_1}G_t^{\alpha_1}\right)(\sqrt{a}x_1,\sqrt{a}y_1) -\left( \partial_{x_1} G_t\right)(\sqrt{a}x_1,\sqrt{a}y_1) \Big| \prod_{j=2}^mW_t^{\alpha_j}(\sqrt{a}x_j,\sqrt{a}y_j)\frac{dt}{\sqrt{t}} \nonumber \\
      & \leq \frac{C}{a |x-y|^{m+1}}, \quad  x,y \in (0,\infty)^m,\;\;x\neq y.
    \end{align}
    Moreover, we apply Lema~\ref{Lem G}, $(g)$, to obtain, for each $x,y \in (0,\infty)^m$, $x \neq y$,
    \begin{align}\label{eqG33}
        a^{m/2}x_1\int_0^{\infty} & \Big| \partial_{x_1}^2\Big( G_t - G_t^{\alpha_1} \Big)(\sqrt{a}x_1,\sqrt{a}y_1) \Big|  \prod_{j=2}^mW_t^{\alpha_j}(\sqrt{a}x_j,\sqrt{a}y_j) \frac{dt}{\sqrt{t}}
         \leq \frac{C}{a|x-y|^{m+1}}
    \end{align}

    By using Lemma~\ref{Lem W}, $(a)$ and $(e)$; and by proceeding as in the previous cases we obtain
    \begin{align}\label{eqG34}
      &  a^{m/2}\int_0^{\infty}  \Big| \left(\partial_{x_1}G_t^{\alpha_1}\right)(\sqrt{a}x_1,\sqrt{a}y_1) -\left( \partial_{x_1} G_t\right)(\sqrt{a}x_1,\sqrt{a}y_1) \Big|
                \sum_{j=2}^m \prod_{i=2,i\neq j}^mW_t(\sqrt{a}x_i,\sqrt{a}y_i) \nonumber  \\
      & \qquad \qquad  \times\Big|x_j \left(\partial_{x_j}W_t^{\alpha_j}\right)(\sqrt{a}x_j,\sqrt{a}y_j)+y_j \left(\partial_{y_j}W_t^{\alpha_j}\right)(\sqrt{a}x_j,\sqrt{a}y_j) \Big| \frac{dt}{\sqrt{t}} \nonumber \\
      & \qquad \leq C a^{(m-1)/2}\int_0^\infty \frac{e^{-ca|x-y|^2/t}}{t^{(m+3)/2}} dt\leq \frac{C}{a |x-y|^{m+1}}, \quad  x,y \in (0,\infty)^m,\;\;x\neq y.
    \end{align}

    Finally, as in the proof of \eqref{eqG33} we get
     \begin{align}\label{eqG35}
        a^{m/2}\int_0^{\infty} & y_1\Big|\Big( \partial_{y_1x_1}^2 G_t^{\alpha_1} \Big)(\sqrt{a}x_1,\sqrt{a}y_1)-\Big( \partial_{y_1x_1}^2 G_t \Big)(\sqrt{a}x_1,\sqrt{a}y_1) \Big|  \prod_{j=2}^mW_t(\sqrt{a}x_j,\sqrt{a}y_j) \frac{dt}{\sqrt{t}} \nonumber \\
            & \leq \frac{C}{a(|x-y|^{m+1})},\quad  x,y \in (0,\infty)^m,\;\;x\neq y.
    \end{align}

    Putting together \eqref{eqG24}, \eqref{eqG27}, \eqref{eqG33}, \eqref{eqG34} and \eqref{eqG35} we deduce
     \begin{align*}
         \Big|\partial_{x_1}  \partial_a \Big( R_{\alpha,1}(x,y;a) - R_1(x,y;a) \Big) \Big| \leq \frac{C}{a|x-y|^{m+1}},\quad  x,y \in (0,\infty)^m,\;\;x\neq y.
    \end{align*}

    If $a<0$ we can proceed in a similar way and \eqref{eqG22} is established.

    According to \cite[Proof of Lemma 2.2 and Proposition 2.3]{JST}, we have that
    $$\sup_{a \in \R \setminus \{0\}} \Big|a {d\over{da}} R_1(x,y;a) \Big|
        \leq \frac{C}{|x-y|^m}, \quad  x,y \in (0,\infty)^m, \ x \neq y, $$
    and
    $$\sup_{a \in \R \setminus \{0\}} \Big\{\Big| \nabla_x \Big(a {d\over{da}} R_1(x,y;a) \Big) \Big| + \Big| \nabla_y \Big(a {d\over{da}} R_1(x,y;a) \Big) \Big| \Big\}
        \leq \frac{C}{|x-y|^{m+1}}, \quad  x,y \in (0,\infty)^m, \ x \neq y. $$
    Then, from \eqref{eqG0} and \eqref{eqG22}, we conclude that
    $$\sup_{a \in \R \setminus \{0\}} \Big| a {d\over{da}}R_{\alpha,1}(x,y;a) \Big|
        \leq \frac{C}{|x-y|^m}, \quad  x,y \in (0,\infty)^m, \ x \neq y, $$
    and
    $$\sup_{a \in \R \setminus \{0\}} \Big\{\Big| \nabla_x \Big(a {d\over{da}} R_{\alpha,1}(x,y;a) \Big) \Big| + \Big| \nabla_y \Big(a {d\over{da}} R_{\alpha,1}(x,y;a) \Big) \Big| \Big\}
        \leq \frac{C}{|x-y|^{m+1}}, \quad  x,y \in (0,\infty)^m, \ x \neq y. $$

        By \cite[Theorem 1.3, Chapter XII]{Tor} the proof of this lemma is finished.
\end{proof}



\end{document}